\pdfoutput=1

\documentclass[preprint,10pt]{elsarticle}
\usepackage[utf8]{inputenc}

\usepackage{geometry}
\usepackage{verbatimbox}

\usepackage{diagbox} 

\usepackage{amsmath, amsfonts, amssymb,mathdots}
\usepackage{amsthm} 
\usepackage{hyperref}
\hypersetup{colorlinks,citecolor=blue}

\usepackage{hypcap}
\usepackage{stmaryrd} 

\usepackage{stackengine}
\stackMath
\newlength\matfield
\newlength\tmplength

\usepackage{graphicx, color, bm}
\usepackage{subfig}
\usepackage{booktabs}
\usepackage{lipsum} 

\usepackage{tikz,pgfplots,pgfplotstable}
\definecolor{markercolor}{RGB}{124.9, 255, 160.65}

\newtheorem{theorem}{Theorem}[section]
\newtheorem{lemma}[theorem]{Lemma}
\newtheorem{remark}{Remark}

\renewcommand{\hat}{\widehat}
\renewcommand{\tilde}{\widetilde}

\newcommand*\diff[1]{\mathop{}\!{\mathrm{d}#1}} 
\newcommand{\diag}[1]{{\rm diag}\LRp{#1}} 
\newcommand{\td}[2]{\frac{{\rm d}#1}{{\rm d}{ {#2}}}}
\newcommand{\pd}[2]{\frac{\partial#1}{\partial#2}}

\newcommand{\LRp}[1]{\left( #1 \right)}
\newcommand{\LRs}[1]{\left[ #1 \right]}
\newcommand{\LRa}[1]{\left\langle #1 \right\rangle}

\newcommand{\LRc}[1]{\left\{ #1 \right\}}

\newcommand{\jump}[1] {\ensuremath{\llbracket#1\rrbracket}}
\newcommand{\avg}[1] {\ensuremath{\LRc{\!\LRc{#1}\!}}}

\newcommand{\fnt}[1]{\bm{\mathsf{ #1}}}

\newcommand{\eq}[1]{\begin{align*}#1\end{align*}}
\newcommand{\eqlab}[1]{\begin{align}#1\end{align}}
\newcommand{\bmat}[1]{\begin{bmatrix}#1\end{bmatrix}}

\graphicspath{{figs/}}

\date{}

\begin{document}


\begin{frontmatter}
\title{Entropy stable modal discontinuous Galerkin schemes and wall boundary conditions for the compressible Navier-Stokes equations}
\author[rice]{Jesse Chan}
\ead{jesse.chan@rice.edu}
\author[rice]{Yimin Lin}
\ead{yiminlin@rice.edu}
\author[vt]{Tim Warburton}
\ead{tcew@vt.edu}
\address[rice]{Department of Computational and Applied Mathematics, Rice University, 6100 Main St, Houston, TX, 77005}
\address[vt]{Department of Mathematics, Virginia Tech, 225 Stanger Street, Blacksburg, VA 24061-1026}
\begin{abstract}
Entropy stable schemes ensure that physically meaningful numerical solutions also satisfy a semi-discrete entropy inequality under appropriate boundary conditions. In this work, we describe a discretization of viscous terms in the compressible Navier-Stokes equations which enables a simple and explicit imposition of entropy stable no-slip (adiabatic and isothermal) and reflective (symmetry) wall boundary conditions for discontinuous Galerkin (DG) discretizations. Numerical results confirm the robustness and accuracy of the proposed approaches.
\end{abstract}
\end{frontmatter}

\section{Introduction}

Computational fluid dynamics (CFD) has relied mainly on first and second order numerical methods, which are robust and reliable. However, because higher order schemes offer improved accuracy at similar computational costs, they have received significant interest as demand for greater resolution in engineering simulations increases \cite{cfd2030}. Discontinuous Galerkin (DG) schemes are among the most popular high order schemes for CFD, especially for transient vorticular flows \cite{wang2013high, huynh2014high}. However, high order methods typically suffer from issues of robustness, especially in the presence of shocks and under-resolved solution features. Entropy stable high order DG schemes \cite{carpenter2014entropy, gassner2016split, chen2017entropy, crean2018entropy, chan2017discretely} provide one way to improve robustness without sacrificing high order accuracy. This improved robustness can be attributed to the fact that entropy stable schemes are stable in the sense that they satisfy a semi-discrete entropy inequality, even in the presence of aliasing errors resulting from under-integration, nonlinear fluxes, and curved geometries \cite{mengaldo2015dealiasing}.

Entropy stable DG schemes for the compressible Euler and Navier-Stokes equations were introduced for tensor product (quadrilateral and hexahedral) meshes by Carpenter et al.\ in \cite{carpenter2014entropy} and Gassner, Winters, and Kopriva in \cite{gassner2016split}. The construction of such schemes utilized connections between nodal DG spectral element methods (DG-SEM) and summation by parts (SBP) finite difference operators. These schemes were later extended to simplicial meshes in \cite{chen2017entropy, crean2018entropy} based on a generalization of SBP operators to the multi-dimensional case \cite{hicken2016multidimensional}. Entropy stable schemes were then extended to more general ``modal'' DG formulations in \cite{chan2017discretely, chan2018discretely, chan2019skew}. Other recent entropy stable numerical schemes include staggered grid schemes \cite{parsani2016entropy, fernandez2019staggered}, collocation schemes based on Gauss points \cite{chan2018efficient, chan2020mortar}, and entropy stable reduced order models \cite{chan2019entropy}. Entropy stable schemes have also been extended to the fully discrete case using entropy conservative and entropy stable relaxation Runge-Kutta time-stepping methods \cite{ranocha2019relaxation, ranocha2020fully}. 

For periodic domains, entropy stable schemes automatically guarantee the satisfaction of a semi-discrete entropy inequality. However, for non-periodic domains, entropy stable schemes must also be paired with appropriate entropy stable boundary conditions. Boundary conditions for DG schemes are typically imposed through the solution of appropriate Riemann problems \cite{mengaldo2014guide}, though not all such boundary conditions are entropy stable. The stability of boundary conditions for the compressible Navier-Stokes equations has typically been analyzed based on a linearized stability analysis \cite{svard2008stable}; however, linearly stable boundary conditions do not necessarily imply entropy stability either. Instead, more recent work has focused on the construction of nonlinearly stable boundary conditions for the compressible Euler and Navier-Stokes equations. Inviscid entropy stable wall and far-field boundary conditions for the compressible Euler equations were investigated in \cite{svard2014entropy, chen2017entropy}, and viscous entropy stable adiabatic wall boundary conditions were analyzed in \cite{parsani2015entropy, svard2018entropy, dalcin2019conservative}. 

In this work, we focus on the construction of viscous wall boundary conditions for the compressible Navier-Stokes equations which mimic the continuous entropy balance. The key novelty of this work is a modified DG discretization of the viscous terms which simplifies methods for imposing viscous wall boundary conditions. In \cite{parsani2015entropy, svard2018entropy, dalcin2019conservative}, viscous wall boundary conditions are imposed by transforming between conservative and primitive variables. In this work, we introduce a modified viscous discretization which is more amenable to modal DG discretizations. We also show this formulation enables the imposition of no-slip wall boundary conditions in a simple and explicit fashion while also providing simpler proofs of entropy conservation. Finally, we derive an entropy stable imposition of reflective symmetry boundary conditions on the viscous stresses, which have not yet been treated in the literature on entropy stable schemes. 

The outline of the paper is as follows: Section~\ref{sec:esthry} reviews entropy stability theory for the compressible Navier-Stokes equations, and Section~\ref{sec:esdg} reviews the construction of entropy stable high order ``modal'' DG methods. Section~\ref{sec:wbc} describes the the imposition of adiabatic, isothermal, no-slip, and symmetry wall boundary conditions which mimic the continuous entropy balance, and discusses the construction of boundary penalization terms. Section~\ref{sec:num} provides numerical experiments which verify our theoretical results, and we provide conclusions and outlook in Section~\ref{sec:conc}.

\section{Entropy stability for the compressible Navier-Stokes equations}
\label{sec:esthry}

Let $\bm{u}$ denote the vector of conservative variables. In two dimensions, these are
\[
\bm{u} = \LRc{\rho, \rho u_1, \ldots, \rho u_d, E} \in \mathbb{R}^{d+2}.
\]
Here, $\rho$ is density, $u_i$ denotes the velocity in the $i$th coordinate direction, and $E$ denotes the specific total energy. We also introduce the pressure $p$ and temperature $T$, which are related to the conservative variables through the constitutive relations
\[
p = (\gamma-1) \rho e, \qquad E = {e+\frac{1}{2} \sum_{i=1}^d u_i^2}, \qquad e = c_v T,
\]
where $\gamma = 1.4$, $e$ is the internal energy density, and $c_v$ is the specific heat at constant volume. ${\rm Pr}$ denotes the Prandtl number, and $\mu, \lambda$ are the dynamic and bulk viscosity coefficients, respectively. 

The compressible Navier-Stokes equations in $d$ dimensions are given by 
\eqlab{
\pd{\bm{u}}{t} + \sum_{i=1}^d \pd{\bm{f}_i}{x_i}  = \sum_{i=1}^d \pd{\bm{g}_i}{x_i} , \label{eq:cns}
}
where $\bm{f}_i$ denote the inviscid fluxes in the $i$th coordinate direction. 

In this work, we focus on the two-dimensional compressible Navier-Stokes equations. However, the main contributions of this paper are straightforward to extend to three dimensions, and we present results in a dimension-independent manner when possible. For $d=2$, the inviscid fluxes $\bm{f}_i$ are given by
\[
\bm{f}_1 = \bmat{
\rho u_1\\
\rho u_1^2 + p\\
\rho u_1u_2 \\
u_1 (E + p)
}, \qquad
\bm{f}_2 = \bmat{
\rho u_2\\
\rho u_1u_2\\
\rho u_2^2 + p\\
u_2 (E + p)
}
\]
The viscous fluxes $\bm{g}_1, \bm{g}_2$ for $d=2$ are given by
\eqlab{
\bm{g}_1 = \bmat{
0\\
\tau_{1,1}\\
\tau_{2,1} \\
\sum_{i=1}^d \tau_{i,1}u_i - \kappa \pd{T}{x_1}
}, \qquad
\bm{g}_2 = \bmat{
0\\
\tau_{1,2}\\
\tau_{2,2} \\
\sum_{i=1}^d \tau_{i,2}u_i - \kappa \pd{T}{x_2}
}.
\label{eq:cnsflux}
}
Here, $\kappa= \kappa(T)$ denotes the thermal conductivity, and $\tau_{i,j}$ denote the components of the viscous stress tensor
\eqlab{
\tau_{i,j} = \mu\LRp{\pd{u_i}{x_j} + \pd{u_j}{x_i}} - \delta_{ij} \lambda \LRp{\sum_{i=1}^d\pd{u_i}{x_i} }, \qquad 1\leq i,j \leq d.
\label{eq:taustress}
}
We assume Stokes hypothesis in this work, or that $\lambda = \frac{2}{3}\mu$. 

\subsection{Nondimensionalization}

We follow \cite{chan2014dpg} and introduce nondimensional quantities for length, density, velocity, temperature, and viscosity
\begin{gather}
\bm{x}^* = \frac{\bm{x}}{L}, \qquad \rho^* = \frac{\rho}{\rho_{\infty}}, \qquad T^* = \frac{T}{T_\infty}, \qquad \mu^* = \frac{\mu}{\mu_{\infty}}\\
u_i = \frac{u_i}{U_{\infty}}, \qquad i = 1,\ldots, d.
\end{gather}
We can then non-dimensionalize pressure, internal energy, and bulk viscosity with respect to combinations of reference quantities
\[
p^* = \frac{p}{\rho_{\infty} U_{\infty}^2}, \qquad e = \frac{e}{U_\infty^2}, \qquad \lambda^* = \frac{\lambda}{\mu_\infty}.
\]
We introduce the Reynolds and free-stream Mach numbers
\eqlab{
{\rm Re} = \frac{\rho_\infty U_\infty L}{\mu_\infty}, \qquad {\rm Ma} = \frac{U_\infty}{\sqrt{\gamma(\gamma-1)c_v T_\infty}}.
\label{eq:ReMa}
}
Note that the reference Mach number is the ratio of the free-stream velocity to the free-stream speed of sound $a_\infty$  
\[
a_\infty = \sqrt{\frac{\gamma p_\infty}{\rho_\infty}} = \sqrt{\gamma (\gamma-1) c_v T_\infty},
\]
since $p = (\gamma-1)\rho e$ and $e = c_v T$. 

The non-dimensionalized equations take the same form as the original equations if we define new physical parameters
\[
\tilde{\mu} = \frac{\mu^*}{{\rm Re}}, \qquad \tilde{\lambda} = \frac{\lambda^*}{{\rm Re}}, \qquad \tilde{c}_v = \frac{1}{\gamma(\gamma-1){\rm Ma}^2}, 
\qquad \tilde{\kappa} = \frac{\gamma \tilde{c}_v \tilde{\mu}}{\rm Pr}.
\]
From this point on, we drop both the tilde and the $*$ superscript and assume all variables to refer to their nondimensionalized quantities. 

\subsection{Entropy variables and symmetrization}

The compressible Navier-Stokes equations admit a mathematical entropy inequality with respect to the convex scalar entropy function $S(\bm{u})$
\[
S(\bm{u}) = -\rho s, 
\]
where $s = \log\LRp{\frac{p}{\rho^\gamma}}$ denotes the physical entropy \cite{hughes1986new}. The derivative of the entropy with respect to the conservative variables yield the entropy variables $\bm{v}(\bm{u}) = \pd{S}{\bm{u}} = \LRc{v_1, v_2, v_3, v_4}$, where
\eqlab{
v_1 = \frac{\rho e (\gamma + 1 - s) - E}{\rho e}, \qquad v_{1+ i}= \frac{\rho {{u}_i}}{\rho e}, \qquad v_{d+2} = -\frac{\rho}{\rho e} \label{eq:evars}
}
for $i = 1,\ldots, d$.  The inverse mapping is given by
\begin{align*}
\rho = -(\rho e) v_{d+2}, \qquad 
\rho {u_i} = (\rho e) v_{1+i}, \qquad 
 E = (\rho e)\LRp{1 - \frac{\sum_{j=1}^d{v_{1+j}^2}}{2 v_{d+2}}},
\end{align*}
where $i = 1,\ldots,d$, and $\rho e$ and $s$ in terms of the entropy variables are 
\begin{equation*}
\rho e = \LRp{\frac{(\gamma-1)}{\LRp{-v_{d+2}}^{\gamma}}}^{1/(\gamma-1)}e^{\frac{-s}{\gamma-1}}, \qquad 
s = \gamma - v_1 + \frac{\sum_{j=1}^d{v_{1+j}^2}}{2v_{d+2}}.
\end{equation*}
It was shown in \cite{hughes1986new} that the entropy variables \textit{symmetrizes} the viscous fluxes in the sense that
\eqlab{
\sum_{i=1}^d \pd{\bm{g}_i}{x_i} = \sum_{i,j=1}^d \pd{}{x_i} \LRp{\bm{K}_{ij}\pd{\bm{v}}{x_j}}. 
\label{eq:symvisc}
}
where $\bm{K}_{ij}$ denote blocks of a symmetric and positive semi-definite matrix $\bm{K}$
\[
\bm{K} = \bmat{\bm{K}_{11} & \ldots & \bm{K}_{1d}\\
\vdots & \ddots & \vdots\\
\bm{K}_{d1} & \ldots & \bm{K}_{dd}} = \bm{K}^T, \qquad \bm{K} \succeq 0.
\]
Formulas for these matrices for $d=2$ are given in terms of the entropy variables and physical parameters 
\begin{gather*}
K_{11} =\frac{1}{v_4^3} 
\left(
    \begin{array}{cccc}
        0 & 0 & 0 & 0 \\
        0 & -(\lambda+2\mu)v_4^2 & 0 & (\lambda+2\mu)v_2v_4 \\ 
        0 & 0 & -\mu v^2_4 & \mu v_3v_4 \\
       0 & (\lambda+2\mu)v_2v_4 & \mu v_3v_4 & -[(\lambda+2\mu)v_2^2 + \mu(v_3^2) -\gamma\mu v_4/Pr]
    \end{array}
\right)\\
K_{12} = \frac{1}{v_4^3}\left(\begin{array}{cccc}
0 & 0 & 0 & 0 \\
0 & 0 & -\lambda v_4^2 & \lambda v_3v_4 \\
0 & -\mu v_4^2 & 0 & \mu v_2v_4 \\
0 & \mu v_3v_4 & \lambda v_2v_4 & (\lambda+\mu)(-v_2v_3) 
\end{array}\right)\\
K_{21} = \frac{1}{v_4^3}\left(\begin{array}{cccc}
0 & 0 & 0 & 0 \\
0 & 0 &  -\mu v_4^2  & \mu v_3v_4 \\
0 & -\lambda v_4^2 & 0 & \lambda v_2v_4 \\
0 & \lambda v_3v_4 & \mu v_2v_4 & (\lambda+\mu)(-v_2v_3)
\end{array}\right)\\
K_{22} = 
\frac{1}{v_4^3} 
\left(
    \begin{array}{cccc}
        0 & 0 & 0 & 0 \\
        0 & -\mu v_4^2 & 0 & \mu v_2v_4 \\ 
        0 & 0 & -(\lambda+2\mu) v^2_4 & (\lambda+2\mu) v_3v_4 \\
       0 & \mu v_2v_4 & (\lambda+2\mu)v_3v_4 & -[(\lambda+2\mu)v_3^2 + \mu(v_2^2) -\gamma\mu v_4/Pr]
    \end{array}
\right)
\end{gather*}
Similar formulas for the symmetrized matrices $\bm{K}_{ij}$ in three-dimensions are derived in \cite{hughes1986new}.

\subsection{Continuous entropy balance}

An entropy balance equation can be derived by multiplying the compressible Navier-Stokes equations by the entropy variables and integrating over the domain. We begin by introducing a few related identities. It can be shown that the following identity is satisfied 
\eqlab{
\bm{v}^T\pd{\bm{f}_i(\bm{u})}{x_i} &= \pd{F_i(\bm{u})}{x_i} \label{eq:fluxidentity}\\
F_i(\bm{u}) &= \bm{v}(\bm{u})^T\bm{f}_i(\bm{u})- \psi_i(\bm{u}), \nonumber
}
where $F_i(\bm{u})$ and $\psi_i(\bm{u})$ denote scalar entropy fluxes and potentials, respectively. For the compressible Navier-Stokes equations, $F_i(\bm{u})$ and $\psi_i(\bm{u})$ are given by \cite{hughes1986new, chen2017entropy}
\[
F_i(\bm{u}) = -\frac{s \rho u_i }{\gamma-1}, \qquad \psi_i(\bm{u}) = \rho u_i.
\]
Multiplying (\ref{eq:cns}) by $\bm{v}^T$, integrating over $\Omega$, and using the chain rule and aforementioned identities then yields
\eqlab{
\int_{\Omega} \pd{S(\bm{u})}{t} + \int_{\partial \Omega} \sum_{i=1}^d \LRp{F_i(\bm{u}) - \bm{v}^T\bm{g}_i}n_i +  \int_{\Omega}\sum_{i,j=1}^d \LRp{\pd{\bm{v}}{x_i}}^T\LRp{\bm{K}_{i,j} \pd{\bm{v}}{x_j}}  = 0.
}
Using that $e = c_vT$, along with definitions of the entropy variables and viscous fluxes $\bm{g}_i$, we can show that the boundary contributions $\bm{v}^T\bm{g}_i$ reduce to a scaling by $c_v$ of the quantity known as ``heat entropy flow'' \cite{dalcin2019conservative}
\eqlab{
\bm{v}^T\bm{g}_i  = \frac{1}{c_v T}\kappa\pd{T}{x_i}.
}
Thus, the entropy balance for the compressible Navier-Stokes equations is
\eqlab{
\label{eq:continentropy1}
\int_{\Omega} \pd{S(\bm{u})}{t} = \int_{\partial \Omega} \sum_{i=1}^d \LRp{\frac{1}{c_vT}\kappa\pd{T}{x_i}-F_i(\bm{u})}n_i -  \int_{\Omega}\sum_{i,j=1}^d \LRp{\pd{\bm{v}}{x_i}}^T\LRp{\bm{K}_{i,j} \pd{\bm{v}}{x_j}}.
}
Since the latter term involving $\bm{K}_{ij}$ is non-positive, we can bound the rate of change of the integrated entropy by
\eqlab{
\int_{\Omega} \pd{S(\bm{u})}{t} \leq \int_{\partial \Omega} \sum_{i=1}^d \LRp{\frac{1}{c_vT}\kappa\pd{T}{x_i}-F_i(\bm{u})}n_i.
\label{eq:entropybalance}
}
For certain boundary conditions, both the inviscid and viscous boundary terms in (\ref{eq:entropybalance}) vanish \cite{svard2014entropy, parsani2015entropy, chen2017entropy,svard2018entropy, dalcin2019conservative}, implying that the solution is entropy stable. More generally, the goal of this work will be to impose boundary conditions such that the semi-discrete entropy inequality mimics the continuous entropy balance (\ref{eq:continentropy1}). 

\section{Entropy stable modal DG discretizations}
\label{sec:esdg}

\subsection{On notation}

The notation in this paper is motivated by notation in \cite{crean2018entropy, fernandez2019entropy}. Unless otherwise specified, vector and matrix quantities are denoted using lower and upper case bold font, respectively.  Spatially discrete quantities are denoted using a bold sans serif font. Finally, the output of continuous functions evaluated over discrete vectors is interpreted as a discrete vector. 

For example, if $\fnt{x}$ denotes a vector of point locations, i.e., $(\fnt{x})_i = \bm{x}_i$, then $u(\fnt{x})$ is interpreted as the vector 
\[	({u}(\fnt{x}))_i = {u}(\bm{x}_i).
\]
Similarly, if $\fnt{u} = {u}(\fnt{x})$, then ${f}(\fnt{u})$ corresponds to the vector
\[
	({f}(\fnt{u}))_i = {f}(u(\bm{x}_i)).
\]
Vector-valued functions are treated similarly. For example, given a vector-valued function $\bm{f}:\mathbb{R}^n\rightarrow \mathbb{R}^n$ and a vector of coordinates $\fnt{x}$, we adopt the convention that $\LRp{\bm{f}(\fnt{x})}_i = \bm{f}(\bm{x}_i)$.

\subsection{Modal DG discretizations}

We now discuss the construction of an entropy stable DG discretization for the compressible Navier-Stokes equations. 
For generality, we assume a ``modal'' framework which is applicable to a broad range of approximation spaces and quadrature rules. We assume the domain $\Omega$ can be decomposed into non-overlapping elements $D^k$, each of which is the image of a reference element $\hat{D}$ under an invertible mapping $\bm{\Phi}^k$. Let $\hat{n}_i$ denote the $i$th component of the outward normal vector on the boundary of the reference element $\partial \hat{D}$, and let $\hat{J}_f$ denote the determinant of the Jacobian of the transformation between a face of $\hat{D}$ and some reference face. Let $\hat{\bm{x}}, \bm{x}$ denote coordinates on the reference element $\hat{D}$ and physical element $D^k$, respectively, such that
\eqlab{
\bm{\hat{x}} = \LRc{\hat{x}_1, \ldots, \hat{x}_d}, \qquad \bm{x} = \LRc{x_1, \ldots, x_d}.
\label{eq:coords}
}
We also assume that the boundary of each element $D^k$ is denoted by $\partial D^k$, and that the outward unit normal on each face in $\partial D^k$ is denoted by $\bm{n} = \LRc{n_1,\ldots,n_d}$. Finally, let $J^k$ denote the determinant of the Jacobian of the mapping $\bm{\Phi^k}$, and let $J^k_f$ denote the determinant of the Jacobian of the mapping from a face of $\partial D^k$ to a reference face. 

Local approximation spaces on each physical element $D^k$ are defined as mappings of a reference approximation space. For this work, we assume $\hat{D}$ is the bi-unit right triangle and that the reference approximation space is the space of total degree $N$ polynomials
\[
P^N = \LRc{\hat{x}_1^i\hat{x}_2^j, \quad i,j \geq 0, \quad i+j\leq N}
\]
where $\hat{x}_i$ denotes the $i$th coordinate on the reference element. 

Next, we introduce notation for jumps and averages of functions across element interfaces. Let $u(\bm{x})$ be a scalar function on $D^k$, and let $u, u^+$ denote its ``interior'' and ``exterior'' values across the face shared by neighbor $D^{k,+}$
\[
\avg{u} = \frac{u^+ + u}{2}, \qquad \jump{u} = u^+- u.
\]
The jump and average of vector-valued functions are defined component-wise. Boundary conditions are also imposed by specifying appropriate exterior values. 

We also assume volume and surface quadrature rules which are exact for degree $2N$ polynomials.  Let $\LRc{\bm{x}_i, w_i}_{i=1}^{N_q}$ denote the points and weights of the volume quadrature rule, and let $\LRc{\bm{x}^f_i, w^f_i}_{i=1}^{N^f_q}$ denote the points and weights of the surface quadrature rule. Now, let $\LRc{\phi_i(\bm{x})}_{i=1}^{N_p}$ denote basis functions for $P^N$. We define the quadrature-based interpolation matrices $\fnt{V}_q, \fnt{V}_f$, mass matrix $\fnt{M}$, and integrated differentiation matrices $\hat{\fnt{\bm{Q}}}^i$
\begin{gather*}
\LRp{\fnt{V}_q}_{ij} = \phi_j(\bm{x}_i), \qquad \LRp{\fnt{V}_f}_{ij} = \phi_j\LRp{\bm{x}^f_i},\\
\fnt{M} = \fnt{V}_q^T\fnt{W}\fnt{V}_q, \qquad \fnt{W} = \diag{\bm{w}}, \qquad (\hat{\fnt{Q}}_{i})_{jk} = \int_{\hat{D}} \pd{\phi_k}{\hat{x}_i} \phi_j.
\end{gather*}
Finally, we introduce inner product notation on an element $D^k$
\[
\LRp{u,v}_{D^k} = \int_{D^k} u(\bm{x})v(\bm{x})\diff{x}, \qquad \LRa{u,v}_{\partial D^k} = \int_{\partial D^k} u(\bm{x})v(\bm{x})\diff{x}
\]
as well as over the entire domain $\Omega$ and its boundary $\partial \Omega$
\[
\LRp{u,v}_{\Omega} = \sum_k \LRp{u,v}_{D^k}, \qquad \LRa{u,v}_{\partial \Omega} = \sum_k \LRa{u,v}_{\partial D^k \cap \partial \Omega}.
\]
In all numerical experiments, integrals are computed via quadrature approximations, which in turn induces discrete $L^2$ inner products which approximate continuous $L^2$ inner products over $D^k, \partial D^k$. Because the following proofs only use properties of quadrature-based $L^2$ inner products and do not assume exact integration, all theoretical results also hold under inexact quadrature.

\subsection{Discretization of inviscid terms}

For most numerical methods, the continuous identity (\ref{eq:fluxidentity}) for the inviscid fluxes does not hold at the semi-discrete level. To address this issue, the inviscid terms are discretized using a ``flux differencing'' approach involving summation-by-parts (SBP) operators and entropy conservative fluxes \cite{tadmor1987numerical}. We briefly review the construction of entropy stable methods for the inviscid case. 

We introduce the quadrature-based projection matrix $\fnt{P}_q= \fnt{M}^{-1}\fnt{V}_q^T\fnt{W}$. Using $\fnt{P}_q$ and $\hat{\fnt{Q}}_i$, we can construct quadrature-based differentiation and extrapolation matrices $\fnt{Q}^i, \fnt{E}$
\[
\fnt{Q}_i = \fnt{P}_q^T\hat{\fnt{Q}}_i\fnt{P}_q, \qquad \fnt{E} = \fnt{V}_f\fnt{P}_q.
\]
To accomodate general quadrature rules (e.g., both with and without boundary points), we introduce hybridized SBP operators. Let 
\[
\fnt{B}_i = \diag{\fnt{w}_f \circ \hat{\fnt{n}}_i}, \qquad \fnt{W}_f = \diag{\fnt{w}_f },
\]
where $\fnt{w}_f$ is a vector of face quadrature points and $\hat{\fnt{n}}_i$ is a vector containing values of the $i$th scaled normal component $\hat{n}_i \hat{J}_f$ at surface quadrature points. Then, the hybridized SBP operator $\fnt{Q}_{i,h}$ on the reference element $\hat{D}$ is defined as
\[
\fnt{Q}_{i,h} = \frac{1}{2}\bmat{
\fnt{Q}_i - \LRp{\fnt{Q}_i}^T & \fnt{E}^T\fnt{B}_i\\
\fnt{B}_i\fnt{E}  & \fnt{B}_i
}.
\]
We can construct operators $\fnt{Q}^k_{i,h}$ on each physical element $D^k$ as follows
\[
\fnt{Q}^k_{i,h} = \sum_{j=1}^d \fnt{G}^k_{ij} \fnt{Q}_{j,h}, 
\]
where $\fnt{G}^k_{ij}$ are diagonal matrices containing the scaled geometric terms $J \pd{\hat{x}_j}{x_i}$ Here, $\hat{x}_j$ and $x_i$ denote the $j$th and $i$th reference and physical coordinates  (\ref{eq:coords}). We also introduce physical boundary matrices 
\[
\fnt{B}^k_i = \fnt{W}_f \diag{\fnt{n}_i \circ \fnt{J}_f^k}, 
\]
where $ \fnt{n}_i, \fnt{J}^k_f$ are vectors containing values of $n_i$ and $J^k_f$ at surface quadrature points. 

We now introduce entropy conservative numerical fluxes $\bm{f}_{i,S}(\bm{u}_L,\bm{u}_R)$ \cite{tadmor1987numerical}, which are bivariate functions of ``left'' and ``right'' states $\bm{u}_L, \bm{u}_R$. In addition to being symmetric and consistent, entropy conservative numerical fluxes satisfy an ``entropy conservation'' property
\eqlab{
\LRp{\bm{v}_L-\bm{v}_R}^T\bm{f}_{i,S}(\bm{u}_L,\bm{u}_R) = \psi_i(\bm{u}_L) - \psi_i(\bm{u}_R).
\label{eq:ecflux}
}

The inviscid flux derivatives are approximated using a ``flux differencing'' approach. We first introduce the $L^2$ projection of the entropy variables and the ``entropy projected'' conservative variables $\tilde{\fnt{u}}$
\[
\fnt{v} = \fnt{P}_q \bm{v}\LRp{\fnt{V}_q\fnt{u}}, \qquad \tilde{\fnt{u}} = \bm{u}\LRp{\fnt{V}_h\fnt{v}}, 
\]
which are defined by evaluating the mapping from entropy to conservative variables using the projected entropy variables. Note that the projected entropy variables $\fnt{v}$ is a vector corresponding to modal coefficients, while $\tilde{\fnt{u}}$ corresponds to point values at volume and face quadrature points. 

Then, $\pd{\bm{f}_i(\bm{u})}{x_i}$ on an element $D^k$ is discretized by 
\[
\pd{\bm{f}_i(\bm{u})}{x_i} \Longleftrightarrow \fnt{V}_h^T\LRp{2\fnt{Q}^k_{i,h} \circ \fnt{F}_i}\fnt{1}, \qquad \LRp{\fnt{F}_i}_{jk} = \bm{f}_{i,S}(\bm{u}_i,\bm{u}_j).
\]
where $\circ$ denotes the matrix Hadamard product. Since the entries of $\fnt{F}_i$ are vector-valued, the Hadamard product $\LRp{2\fnt{Q}^k_{i,h} \circ \fnt{F}_i}$ should be understood as each scalar entry of $2\fnt{Q}^k_{i,h}$ multiplying each component of each vector-valued entry of ${\fnt{F}_i}$.

Finally, let $\tilde{\fnt{u}}^+$ denote the values of $\tilde{\fnt{u}}$ on a neighboring element $D^{k,+}$. The inviscid discretization is completed by specifying interface fluxes which couple neighboring elements together, such that an entropy stable inviscid scheme over each element $D^k$ is
\[
\fnt{M}\td{\fnt{u}}{t} + \sum_{i=1}^d \LRs{\fnt{V}_h^T\LRp{2\fnt{Q}^k_{i,h} \circ \fnt{F}_i}\fnt{1} + \fnt{V}_f^T \LRp{\fnt{B}^k_i \LRp{\bm{f}_{i,S}\LRp{\tilde{\fnt{u}}^+,\tilde{\fnt{u}}} - \bm{f}_i(\fnt{u})}}}  - \fnt{V}_f^T \fnt{W}_f\frac{\lambda}{2}\jump{\tilde{\fnt{u}}}= 0.
\]
Here, we have added a simple entropy dissipative Lax-Friedrichs penalization term, where $\lambda$ is the maximum of the wavespeed between the exterior and interior solution states $\tilde{\fnt{u}}^+$ and $\tilde{\fnt{u}}$. Other penalization terms such as HLLC and certain matrix penalizations \cite{chen2017entropy, winters2017uniquely} also dissipate entropy. 

All that remains for the implementation of the scheme is to specify the entropy conservative numerical fluxes $\bm{f}_S\LRp{\bm{u}_L,\bm{u}_R}$. All experiments in this paper utilize the entropy conservative and kinetic energy preserving numerical fluxes of Chandrashekar \cite{chandrashekar2013kinetic}. These fluxes utilize the logarithmic mean, which is computed in a numerically stable manner using the expansion derived in \cite{winters2019entropy}. 

\begin{remark}
While we have presented entropy stable DG schemes using a general ``modal'' DG framework, the formulation reduces to existing methods under appropriate choices of quadrature and basis. For example, specifying Gauss-Lobatto quadrature on a tensor product element recovers entropy stable spectral collocation schemes \cite{chan2018efficient}. SBP discretizations without an underlying basis on simplices \cite{hicken2016multidimensional, chen2017entropy, crean2018entropy} can also be recovered for appropriate quadrature rules by setting $\fnt{V}_q\fnt{P}_q = \fnt{I}$ \cite{wu2020high}.
\end{remark}

\subsubsection{Entropy stable imposition of inviscid wall conditions}
\label{sec:inviscidBCs}

In the inviscid case, no-slip (no normal flow) boundary conditions are imposed at solid walls \cite{svard2014entropy, chen2017entropy}. These boundary conditions are consistent with all wall boundary conditions considered in this paper, and are imposed by enforcing 
\eqlab{
\rho^+ = \rho, \qquad {u}_n^+ = -u_n, \qquad u_\tau = u_\tau, \qquad p^+ = p.
\label{eq:invisicdwalls}
}
where $u_n, u_\tau$ denote the normal and tangential components of the velocity. Explicit expressions for $u_n, u_\tau$ in 2D are given by
\eq{
u_n &= u_1n_1 + u_2n_2\\
u_\tau &= u_1n_2 - u_2n_1. 
}
It was shown in \cite{svard2014entropy, chen2017entropy} that boundary contributions to the entropy balance equation (\ref{eq:entropybalance}) vanish under the imposition of reflective boundary conditions (\ref{eq:invisicidwalls}). For the remainder of this paper, we will assume that all viscous wall boundary conditions are paired with (and consistent with) these inviscid wall conditions.

\subsection{Discretization of viscous terms}

We discretize the symmetrized viscous terms (\ref{eq:symvisc}) using a local DG formulation \cite{cockburn1998local, zakerzadeh2017entropy}, which is similar to the formulations introduced for nonlinear elliptic PDEs in \cite{cockburn1999some}. We note that, while we presented the inviscid discretization using matrix notation, we utilize a variational formulation more familiar to finite element methods to describe the discretization of the viscous terms. 

We begin by introducing $\bm{\Theta}$, which are DG approximations of the gradients of the entropy variables.  Let $\bm{w}_{1,i} \in \LRs{P^N\LRp{\hat{D}}}^{4}$ denote vector-valued test functions for $i = 1,\ldots, d$. The variational definition of $\bm{\Theta}$ is then given by 
\eqlab{
\LRp{\bm{\Theta}_i, \bm{w}_{1,i}}_{D^k} = \LRp{ \pd{\bm{v}}{x_i},\bm{w}_{1,i}}_{D^k} + \frac{1}{2}\LRa{\jump{\bm{v}}n_{i},\bm{w}_{1,i}}_{\partial D^k}, \qquad i = 1,\ldots,d \label{eq:ldg1}.
}
The terms $\bm{\Theta}_1, \bm{\Theta}_2$ are approximations of derivatives with respect to $x_1, x_2$ of the entropy variables $\bm{v}_i$. In the next step, we compute $\bm{\sigma}_i$  as the $L^2$ projection of $\sum_{j=1}^d \bm{K}_{ij}\bm{\Theta}_j$ for $i = 1,2$ onto the approximation space of each element
\eqlab{
\LRp{\bm{\sigma}_i,\bm{w}_{2,i}}_{D^k} = \LRp{\sum_{j=1}^d\bm{K}_{ij}\bm{\Theta}_j, \bm{w}_{2,i}}_{D^k}, \qquad i = 1,\ldots, d \label{eq:ldg2}
}
for all $\bm{w}_{2,i} \in \LRs{P^N\LRp{\hat{D}}}^4$. Note that $\bm{\sigma}_i$ is an approximation to the viscous flux functions $\bm{g}_i$ in the compressible Navier-Stokes equations (\ref{eq:cns}) and (\ref{eq:cnsflux}). 

We can now approximate the divergence of $\bm{\sigma}$ via $\bm{g}_{\rm visc}$. 
Let $\bm{\tau}_{\rm visc}$ be a positive semi-definite penalty matrix which is single-valued over each element interface, which we will specify later. Then, the divergence of the viscous fluxes is approximated by $\bm{g}_{\rm visc}$ as
\eqlab{
\LRp{\bm{g}_{\rm visc}, \bm{w}_{3}}_{D^k} 
= \sum_{i=1}^d \LRs{\LRp{ -\bm{\sigma}_i, \pd{\bm{w}_{3}}{x_i}}_{D^k} + \LRa{\avg{\bm{\sigma}_i}n_{i},\bm{w}_{3}}_{\partial D^k}} - \LRa{\bm{\tau}_{\rm visc}\jump{\bm{v}},\bm{w}_3}_{\partial D^k},   \label{eq:ldg3}
}
for $i = 1,\ldots,d$ and for all $\bm{w}_3 \in \LRs{P^N(\hat{D})}^4$. This approximation can be shown to be positive semi-definite in the following sense
\begin{lemma}
\label{lemma:visc}
Let $\bm{g}_{\rm visc}$ be defined by (\ref{eq:ldg1}), (\ref{eq:ldg2}), and (\ref{eq:ldg3}). For periodic boundary conditions, the viscous entropy dissipation satisfies $\LRp{\bm{g}_{\rm visc}, \bm{v}} \leq 0$. 
\end{lemma}
\begin{proof}
The proof is similar to those of \cite{cockburn1999some, bustinza2004local, zakerzadeh2017entropy}. Let $\bm{w}_3 = \bm{v}$, $\bm{w}_{1,i} = \bm{\sigma}_i$, and $\bm{w}_2 = \bm{\theta}$. Then, summing up (\ref{eq:ldg1}), (\ref{eq:ldg2}) and using  (\ref{eq:ldg3}) yield
\eq{
\LRp{\bm{g}_{\rm visc}, \bm{v}}_{D^k} &= \sum_{i=1}^d \LRs{\LRp{ -\bm{\sigma}_i, \pd{\bm{v}}{x_i}}_{D^k} + \LRa{\avg{\bm{\sigma}_i}n_{i},\bm{v}}_{\partial D^k}} + \LRa{\bm{\tau}_{\rm visc}\jump{\bm{v}},\bm{v}}_{\partial D^k}\\
\sum_{i=1}^d\LRp{\bm{\Theta}_i, \bm{\sigma}_i}_{D^k} &= \sum_{i=1}^d\LRp{ \pd{\bm{v}}{x_i},\bm{\sigma}_i}_{D^k} + \frac{1}{2}\LRa{\jump{\bm{v}}n_{i},\bm{\sigma}_i}_{\partial D^k} ,\\
\sum_{i=1}^d \LRp{\bm{\sigma}_i,\bm{\Theta}_i}_{D^k} &= \sum_{j=1}^d \LRp{\bm{K}_{ij}\bm{\Theta}_j, \bm{\Theta}_i}_{D^k}.
}
We sum over all elements $D^k$, substitute the second equation into the first one, and use the third equation to yield
\eqlab{
\LRp{\bm{g}_{\rm visc}, \bm{v}}_{D^k} = \sum_{i,j=1}^d \LRs{-\LRp{\bm{K}_{ij}\bm{\Theta}_j, \bm{\Theta}_i}_{D^k}  + \frac{1}{2}\LRa{\jump{\bm{v}}n_{j},\bm{\sigma}_j}_{\partial D^k} +  \LRa{\avg{\bm{\sigma}_j}n_{j},\bm{v}}_{\partial D^k}} + \LRa{\bm{\tau}_{\rm visc}\jump{\bm{v}},\bm{v}}_{\partial D^k}. \label{eq:entropydiss}
}
What remains is to show that the surface terms vanish when summed up over all elements. For periodic boundary conditions, all faces are ``interior'' faces shared by two elements. We split contributions from each surface term and swap them between $D^k$ and the neighboring element $D^{k,+}$, such that 
\eq{
\sum_k \frac{1}{2}\LRa{\jump{\bm{v}}n_{i},\bm{\sigma}_i}_{\partial D^k} &= \frac{1}{2} \sum_k \LRp{\frac{1}{2}\LRa{\jump{\bm{v}}n_{i},\bm{\sigma}_i}_{\partial D^k} + \frac{1}{2}\LRa{\jump{\bm{v}}n_{i},\bm{\sigma}_i^+}_{\partial D^{k,+}}} = \frac{1}{2} \sum_k \LRa{\jump{\bm{v}}n_{i},\avg{\bm{\sigma}_i}}_{\partial D^k}\\
\sum_k \LRa{\avg{\bm{\sigma}_i}n_{i},\bm{v}}_{\partial D^k} &= \frac{1}{2} \sum_k \LRp{\LRa{\avg{\bm{\sigma}_i}n_{i},\bm{v}}_{\partial D^k} - \LRa{\avg{\bm{\sigma}_i}n_{i},\bm{v}^+}_{\partial D^{k,+}}}= -\frac{1}{2} \sum_k\LRa{\avg{\bm{\sigma}_i}n_i, \jump{\bm{v}}}_{\partial D^k}\\
\sum_k  \LRa{\bm{\tau}_{\rm visc}\jump{\bm{v}},\bm{v}}_{\partial D^k} &= \frac{1}{2} \sum_k \LRp{ \LRa{\bm{\tau}_{\rm visc}\jump{\bm{v}},\bm{v}}_{\partial D^k} - \LRa{\bm{\tau}_{\rm visc}\jump{\bm{v}},\bm{v}^+}_{\partial D^{k,+}}}=   -\frac{1}{2}\LRa{\bm{\tau}_{\rm visc}\jump{\bm{v}},\jump{\bm{v}}}_{\partial D^k}
}
Here, we have used that both $n_i$ and $\jump{\bm{v}}$ change sign between $D^k$ and $D^{k,+}$. Thus, the surface terms cancel, and by the positive semi-definiteness of $\bm{K}_{ij}$, 
\[
\sum_k \LRp{\bm{g}_{\rm visc}, \bm{v}}_{D^k} = \sum_k \sum_{i,j=1}^d -\LRp{\bm{K}_{ij}\bm{\Theta}_j, \bm{\Theta}_i}_{D^k} -\frac{1}{2}\LRa{\bm{\tau}_{\rm visc}\jump{\bm{v}},\jump{\bm{v}}}_{\partial D^k} \leq 0
\]
since $\bm{\tau}_{\rm visc}$ is a positive semi-definite matrix.
\end{proof}

\begin{remark}
This proof uses only properties of $L^2$ inner products, which are preserved if all inner products and integals are computed using quadrature. The negative semi-definite structure is also preserved on curved meshes with spatially varying geometric terms if either an SBP property holds or if the derivative lies on the test function in (\ref{eq:ldg1}) and the derivative lies on $\bm{\sigma}_i$ in (\ref{eq:ldg3}). In the former case, the extension to curved meshes is essentially the same as in \cite{fernandez2019extension}, while for the latter case the treatment of viscous terms resembles that of the ``strong-weak'' formulation for DG discretizations of symmetric wave equations \cite{warburton2013low}.
\end{remark}

\section{Entropy stable imposition of wall boundary conditions}
\label{sec:wbc}

We now turn our focus to the entropy stable imposition of adiabatic and isothermal no-slip wall boundary conditions for the compressible Navier-Stokes equations, as well as the entropy stable treatment of slip boundary conditions. Boundary conditions are imposed by choosing appropriate exterior states $\tilde{\bm{u}}^+, \bm{v}^+$ such that the contributions from the boundary terms in the proof of Lemma~\ref{lemma:visc} reduce to appropriate quantities \cite{arnold2002unified, hesthaven2007nodal, brenner2007mathematical}. 

Let $\LRa{u,v}_{\partial\Omega} = \int_{\partial\Omega} uv$ denote the inner product on the domain boundary $\partial \Omega$. For the following proofs we will assume that $\bm{\tau}_{\rm visc} = 0$ on $\partial \Omega$, and postpone the discussion of entropy-dissipative boundary penalization matrices to Section~\ref{sec:bpenmat}. Then, the total viscous entropy contribution is 
\eqlab{
\sum_k \LRp{\bm{g}_{\rm visc}, \bm{v}}_{D^k} = \LRp{\sum_k \sum_{i,j=1}^d -\LRp{\bm{K}_{ij}\bm{\Theta}_j, \bm{\Theta}_i}_{D^k}} + \sum_{i=1}^d \LRs{\frac{1}{2} \LRa{\jump{\bm{v}}n_{i},\bm{\sigma}_i}_{\partial\Omega} +  \LRa{\avg{\bm{\sigma}_i}n_i, \bm{v}}_{\partial\Omega}}. \label{eq:bterms}
}
Our goal will be to construct exterior states for which the discrete viscous entropy-dissipative terms (\ref{eq:bterms}) mimic the continuous viscous entropy dissipative terms in (\ref{eq:continentropy1}).

In the following sections, , we will refer to individual components of the viscous fluxes $\bm{\sigma}_i$ by
\[
\LRp{\bm{\sigma}_i}_j = \sigma_{j,i}, \qquad j = 1,\ldots, d,
\]
for consistency with the $\tau_{i,j}$ notation in (\ref{eq:taustress}). We will also restrict ourselves to the two-dimensional case $d=2$ for simplicity of presentation. Recall that the conservative variables are $\rho, u_1, u_2, E$, and the entropy variables correspond to
\[
v_1 = \frac{\rho e (\gamma + 1 - s) - E}{\rho e}, \qquad v_{2}= \frac{\rho {{u}_1}}{\rho e}, \qquad v_{3}= \frac{\rho {{u}_2}}{\rho e}, \qquad v_{4} = -\frac{\rho}{\rho e}.
\]
The extension to $d=3$ involves straightforward modifications to account for the $z$-component of the normal vector and velocity vector.

\subsection{Adiabatic no-slip wall boundary conditions}

Adiabatic no-slip wall conditions impose zero normal velocity conditions, velocity conditions, and an ``entropy flow'' condition on the temperature gradient through the wall
\[
u_n = 0, \qquad \bm{u}_{\tau} = \bm{u}_{\rm wall}, \qquad \kappa\pd{T}{n}\frac{1}{T} = g(t). 
\]
where $u_n = u_1n_1 + u_2n_2$ and $u_{\tau} = u_1n_2 - u_2n_1$ in 2D. For simplicity of notation, we will convert boundary conditions on normal and tangential components to boundary conditions on velocity in each coordinate direction
\[
u_i = u_{i,{\rm wall}}, \qquad i = 1,\ldots, d.
\]

The terms which naturally appear in the DG formulation involve only traces of entropy variables and approximations of the viscous fluxes. However, we can impose no-slip velocity conditions by noting that the entropy variables $v_{2},\ldots,v_{1+d}$ in (\ref{eq:evars}) are the components of the velocity $u_i$ scaled by $e^{-1}$, and that $v_{4} = -1/e$. Then, the velocity boundary conditions can equivalently be imposed as
\[
v_{1+i} = \frac{u_{i,\rm wall} }{e} = -u_{i,\rm wall} v_{4}, \qquad i = 1,\ldots, d.
\]
We impose these conditions by specifying the exterior states
\eqlab{
v_{1+i}^+ = -2u_{i,\rm wall} v_{4} -v_{1+i}, \qquad i = 1,\ldots, d \label{eq:vbc}
}
such that $\avg{v_{1+i}} = u_{i,\rm wall} v_{4}$. 

We now consider the adiabatic wall condition. Note that the variables $\bm{\sigma}_i$ in (\ref{eq:ldg2}) are approximations to the viscous fluxes $\bm{g}_i$ in (\ref{eq:cnsflux}), which include the heat flux in the last component of $\bm{g}_i$. In two dimensions, the definitions of $\sigma_{i,j}$ correspond to
\eq{
\sigma_{2,i} &= \tau_{1,i}\\
\sigma_{3,i} &= \tau_{2,i}\\
\sigma_{4,i} &= \tau_{1,i}u_1+\tau_{2,i}u_2 - \kappa \pd{T}{x_i}, \qquad i = 1,\ldots,2.
}
We impose adiabatic wall boundary conditions for by specifying $\sigma_{4,i}^+$ as
\eqlab{
\sigma_{4,i}^+ = 2 \LRp{u_{1,{\rm wall}} \sigma_{2,i} + u_{2,{\rm wall}} \sigma_{3,i} + \frac{c_vg(t)n_i}{v_4}} - \sigma_{4,i} , \label{eq:bc4}
}
such that the average of $\sigma_{4,i}$ incorporates wall velocities and heat entropy flow into the formula for the viscous energy flux
\[
\avg{\sigma_{4,i}} = u_{1,{\rm wall}} \sigma_{2,i} + u_{2,{\rm wall}} \sigma_{3,i} + \frac{c_vg(t)n_i}{v_4}.
\]

Finally, since no boundary conditions are imposed on $\sigma_{2,i}, \sigma_{3,i}$, and $v_4$, we simply set the exterior values equal to the interior values for $i = 1,2$
\eqlab{
\sigma_{2,i}^+ = \sigma_{2,i}, \qquad \sigma_{3,i}^+ = \sigma_{3,i}, \qquad v_4^+ = v_4
\label{eq:bc_extrap}
}
such that the average quantities are $\avg{\sigma_{j,i}} = \sigma_{j,i}$  for $j=2, 3$ and $\avg{v_4} = v_4$. Note that $\bm{v}_1^+$ can be arbitrarily chosen since $\sigma_{1,i} = 0$ due to the fact that the corresponding rows of $\bm{K}_{ij}$ are zero. Based on these exterior states, we have the following theorem: 
\begin{theorem}
\label{thm:adiabatic}
Let $\bm{g}_{\rm visc}$ denote viscous contributions from (\ref{eq:ldg1}), (\ref{eq:ldg2}), and (\ref{eq:ldg3}). If adiabatic no-slip wall boundary conditions are imposed using exterior states for $i = 1,\ldots,d$ in $d=2$ dimensions
\eq{
v_{1+i}^+ &= -2u_{i,\rm wall} v_4 -v_{1+i}\\
v_4^+ &= v_4\\
\sigma_{2,i}^+ &= \sigma_{2,i},\\ 
\sigma_{3,i}^+ &= \sigma_{3,i},\\
\sigma_{4,i}^+ &= 2 \LRp{u_{1,{\rm wall}} \sigma_{2,i} + u_{2,{\rm wall}} \sigma_{3,i} + \frac{c_vg(t)n_i}{v_4}} - \sigma_{4,i} ,
}
then the viscous contribution $\bm{g}_{\rm visc}$ mimics the entropy balance such that
\[
\sum_k \LRp{\bm{g}_{\rm visc},\bm{v}}_{D^k} = \sum_{i=1}^d\LRa{c_v g(t),1}_{\partial \Omega} - \sum_k \LRp{\sum_{i,j=1}^d \LRp{\bm{K}_{ij}\bm{\Theta}_j, \bm{\Theta}_i}_{D^k}}.
\]
\end{theorem}
\begin{proof}
Plugging the exterior values into the boundary terms in (\ref{eq:bterms}) simplify to
\eq{
\sum_{i=1}^d \LRs{\frac{1}{2} \LRa{\jump{\bm{v}}n_{i},\bm{\sigma}_i}_{\partial\Omega} +  \LRa{\avg{\bm{\sigma}_i}n_i, \bm{v}}_{\partial\Omega}} &= 
\sum_{i=1}^d \LRa{-u_{1,{\rm wall}}v_4 - v_2, \sigma_{2,i}n_i}_{\partial\Omega} + \LRa{\sigma_{2,i} n_i, v_2}_{\partial\Omega} \\
&+\sum_{i=1}^d\LRa{-u_{2,{\rm wall}}v_4 - v_3, \sigma_{3,i}n_i}_{\partial\Omega} + \LRa{\sigma_{3,i} n_i, v_3}_{\partial\Omega} \\
&+\sum_{i=1}^d\LRa{u_{1,{\rm wall}} \sigma_{2,i} + u_{2,{\rm wall}} \sigma_{3,i} + \frac{c_vg(t)n_i}{v_4}, v_4 n_i}_{\partial\Omega}\\
&= \LRa{c_vg(t),1}_{\partial\Omega}.
}
\end{proof}

As noted in \cite{dalcin2019conservative}, if $g(t) = 0$, then the boundary term resulting from Theorem~\ref{thm:adiabatic} vanishes and the resulting discretization is entropy stable. 

\subsection{Isothermal no-slip wall conditions}
Isothermal no-slip wall boundary conditions impose tangential wall velocity conditions and a fixed temperature at the wall
\[
u_i = u_{i,{\rm wall}}, \quad i = 1,\ldots, d, \qquad T = T_{\rm wall}. 
\]
To impose $T = T_{\rm wall}$, we use that $v_4 = -1/e = -1/(c_vT)$ and set the exterior state $v_4^+$ as
\[
v_4^+ = -\frac{2}{c_vT_{\rm wall}} - v_4, \qquad \sigma_{4,i}^+ = \sigma_{4,i}, \quad i = 1,2,
\]
such that $\avg{v_4} = -1/(c_vT_{\rm wall})$ and $\avg{\sigma_{4,i}} = \sigma_{4,i}$. We also incorporate $T_{\rm wall}$ into the exterior values $v_2^+, v_3^+$. We have the following theorem on entropy stability of isothermal wall boundary conditions:
\begin{theorem}
\label{thm:isothermal}
Let $\bm{g}_{\rm visc}$ denote viscous contributions from (\ref{eq:ldg1}), (\ref{eq:ldg2}), and (\ref{eq:ldg3}). If isothermal no-slip wall boundary conditions are imposed by setting the exterior states for $i = 1,\ldots,d$ in $d=2$ dimensions
\eq{
v_{1+i}^+ &= \frac{2 u_{i,\rm wall}}{c_vT_{\rm wall}}  -v_{1+i}\\
v_4^+ &= -\frac{2}{c_vT_{\rm wall}} - v_4\\
\sigma_{2,i}^+ &= \sigma_{2,i},\\ 
\sigma_{3,i}^+ &= \sigma_{3,i},\\
\sigma_{4,i}^+ &= \sigma_{4,i},
}
then the viscous contribution $\bm{g}_{\rm visc}$ mimics the entropy balance such that
\[
\sum_k \LRp{\bm{g}_{\rm visc},\bm{v}}_{D^k} = \sum_{i=1}^d\LRa{\frac{q_n}{c_vT_{\rm wall}},1}_{\partial \Omega} - \sum_k \LRp{\sum_{i,j=1}^d \LRp{\bm{K}_{ij}\bm{\Theta}_j, \bm{\Theta}_i}_{D^k}}.
\]
Here, we have introduced the normal heat flux $q_n = \sum_{i=1}^d q_in_i$, where $q_i$ is defined as
\[
q_i = -\sigma_{4,i}+u_{1,{\rm wall}} \sigma_{2,i}+u_{2,{\rm wall}} \sigma_{3,i} \approx -\kappa\pd{T}{x_i}.
\]
\end{theorem}
\begin{proof}
Under this choice of exterior states, the boundary terms in (\ref{eq:bterms}) simplify to
\eq{
\sum_{i=1}^d \LRs{\frac{1}{2} \LRa{\jump{\bm{v}}n_{i},\bm{\sigma}_in_i}_{\partial\Omega} +  \LRa{\avg{\bm{\sigma}_i}n_i, \bm{v}}_{\partial\Omega}} &= 
\sum_{i=1}^d \LRa{\frac{u_{1,\rm wall}}{c_vT_{\rm wall}} - v_2, \sigma_{2,i}n_i}_{\partial\Omega} + \LRa{\sigma_{2,i} n_i, v_2}_{\partial\Omega} \\
&+\sum_{i=1}^d\LRa{\frac{u_{2,\rm wall}}{c_vT_{\rm wall}} - v_3, \sigma_{3,i}n_i}_{\partial\Omega} + \LRa{\sigma_{3,i} n_i, v_3}_{\partial\Omega} \\
&+\sum_{i=1}^d\LRa{-\frac{1}{c_vT_{\rm wall}} - v_4, \sigma_{4,i} n_i}_{\partial\Omega} + \LRa{\sigma_{4,i}, v_4 n_i}_{\partial\Omega}\\
&= \sum_{i=1}^d \LRa{\underbrace{-\sigma_{4,i}+u_{1,{\rm wall}} \sigma_{2,i}+u_{2,{\rm wall}} \sigma_{3,i}}_{q_i}, \frac{1}{c_vT_{\rm wall}}n_i}_{\partial\Omega}\\
&=  \LRa{\frac{q_n}{c_vT_{\rm wall}}, 1}_{\partial\Omega}.
}
\end{proof}
Since the wall temperature $T_{\rm wall}$ is assumed to be positive, this boundary contribution does not vanish and the resulting discretization cannot be proven to be entropy stable. However, the boundary contribution mimics the boundary terms in the continuous entropy balance equation (\ref{eq:continentropy1}), which do not vanish at isothermal walls. 

\subsection{Reflective wall (symmetry) conditions}

Finally, we consider reflective boundary conditions, which are the viscous extension of the reflective wall boundary conditions for the inviscid case of the compressible Euler equations \cite{svard2014entropy, chen2017entropy, hindenlang2019stability}. In the context of viscous flows, these boundary conditions can be used to enforce symmetry conditions or free surfaces. Recall from Section~\ref{sec:inviscidBCs} that inviscid reflective wall boundary conditions are enforced by setting exterior values for the convective flux  
\[
\rho^+ = \rho, \qquad {u}_n^+ = -u_n, \qquad u_\tau = u_\tau, \qquad p^+ = p
\]
We note that these conditions correspond to continuous boundary conditions on the normal velocity and normal heat flux 
\eqlab{
u_n = 0, \qquad \kappa\pd{T}{n} = 0.
\label{eq:continBCs}
}

Let $n_1, n_2$ denote the components of the unit normal vector. The conditions on normal velocity imply that the velocity reduces to its tangential component. This is enforced by setting
\[
{u_i}^+ = u_i - 2u_n n_i,\qquad i = 2, \ldots, d+1,
\]
such that $\sum_{i=1}^d\avg{u_i} n_i = 0$ and the normal component of the averaged velocity vanishes. 

Since $e = c_vT > 0$ for $T > 0$, the second and third entropy variables $v_2,v_3 = u/e, v/e$ are well-defined. Thus, reflective wall boundary conditions are also equivalent to the following conditions on the second and third entropy variables
\begin{equation}
{v_i}^+ = v_i - 2v_n n_i, 
\label{eq:bcV}
\end{equation}
where $v_n = v_2 n_1 + v_3 n_2$ in 2D. 

We now consider viscous contributions. Note that $\avg{{\sigma}_{1,j}} = 0$ since there is no mass diffusion, and terms involving ${\sigma}_1$ vanish. We thus begin by considering fields corresponding to $i = 2,3$. Using (\ref{eq:bcV}), the boundary terms involving $\jump{{v}_i}$ for $i=2,3$ in (\ref{eq:bterms}) can be expanded out as
\begin{gather*}
\sum_{i=1,2} \LRa{\frac{1}{2}\jump{{v}_{1+i}}, \sum_{j=1}^d{\sigma}_{1+i,j} n_j} 
= \frac{1}{2}\LRa{{v}_n, \sum_{i,j=1}^d{\sigma}_{1+i,j} n_i n_j}_{\partial \Omega}
= \frac{1}{2}\LRa{v_2n_1 + v_3n_2, \sum_{i,j=1}^d{\sigma}_{1+i,j} n_i n_j}_{\partial \Omega}.
\end{gather*}
We can write this in matrix form using the unit normal vector $\bm{n}= [n_1, n_2]^T$ 
\eqlab{
 \frac{1}{2}\LRa{v_2n_1 + v_3n_2, \sum_{i,j=1}^d{\sigma}_{1+i,j} n_i n_j} = \frac{1}{2}\LRa{\bmat{v_2\\v_3} \cdot \bm{n}, \bm{n}^T \LRp{\bmat{\sigma_{2,1} & \sigma_{2,2}\\ \sigma_{3,1} & \sigma_{3,2}} \bm{n}}}_{\partial \Omega} 
 \label{eq:v23contrib}
}
Recall that in 2D, the boundary contributions involving $\avg{\sigma_{i,j}}$ are
\[
\sum_{i=1,2} \LRa{\sum_{j=1}^d \avg{\sigma_{1+i,j}} n_j, v_{1+i}}_{\partial \Omega}  = \frac{1}{2}\LRa{\bmat{v_2\\v_3}, \bmat{\avg{\sigma_{2,1}} & \avg{\sigma_{2,2}} \\ \avg{\sigma_{3,1}} & \avg{\sigma_{3,2}}} \bm{n}}_{\partial \Omega}.
\]
These contributions will cancel with (\ref{eq:v23contrib}) if the tangential-normal component of the stress vanishes. This condition is equivalent to the stress on the boundary reducing to the normal-normal component 
\[
\bmat{\avg{\sigma_{2,1}} & \avg{\sigma_{2,2}} \\ \avg{\sigma_{3,1}} & \avg{\sigma_{3,2}}} \bm{n}
 = \bm{n}\bm{n}^T {\bmat{\sigma_{2,1} & \sigma_{2,2}\\ \sigma_{3,1} & \sigma_{3,2}} \bm{n}}.
\]
This implies that  $\avg{\sigma_{1+i,j}}$ can be expressed in terms of the normal stress
\eqlab{
\avg{\sigma_{1+i,j}} &= n_j \sigma_{n,j}\\
\sigma_{n,j} &= \sum_{i=1}^d\sigma_{1+i,j}n_i, \qquad i = 1,\ldots, d \label{eq:normalstress}.
}
Finally, we consider contributions involving $v_4$ and $\sigma_{4,i}$
\eqlab{
\sum_{j=1}^d \LRa{\avg{\sigma_{4,j}}n_j,v_4}_{\partial \Omega} + \LRa{\frac{1}{2}\jump{v_4},\sigma_{4,j}n_j}_{\partial \Omega}.
\label{eq:bcsym}
}
Since $v_4 = -1/e$ and $\rho^+, p^+ = \rho, p$ from the inviscid wall boundary conditions, we set the exterior state ${v}_4^+ = {v}_4$.  The remaining boundary term in (\ref{eq:bcsym}) vanishes as well if also we impose ${\sigma}_{4,i}^+ = -{\sigma}_{4,i}$. We note that this corresponds to a zero normal heat flux condition $\kappa\pd{T}{n} = 0$. Recall from (\ref{eq:bc4}) that $\sigma_{4,i}$ corresponds to $\sigma_{4,i} = u_1 \sigma_{2,i} + u_2 \sigma_{3,i} - \kappa\pd{T}{x_i}$, such that 
\[
\sum_{j=1}^d \sigma_{4,j}n_j = \LRp{\sum_{i,j=1}^d u_i \sigma_{i+1,j}n_j} - \kappa\pd{T}{n}.
\]
Recall from (\ref{eq:normalstress}) that the stress reduces to the normal-normal stress on the boundary. We rewrite this using the velocity vector $\bm{u} = [u_1, u_2]^T$ such that 
\eq{
\sum_{i,j=1}^d u_i \sigma_{i+1,j}n_j 
= \bmat{u_1\\u_2}^T \bm{n}\bm{n}^T
\bmat{\sigma_{2,1} & \sigma_{2,2}\\
\sigma_{3,1} & \sigma_{3,2}}\bm{n}
= \LRp{\bm{u}^T\bm{n}}\LRp{\bm{n}^T
\bmat{\sigma_{2,1} & \sigma_{2,2}\\
\sigma_{3,1} & \sigma_{3,2}}\bm{n}} = 0
}
since $\LRp{\bm{u}^T\bm{n}} = u_n = 0$ by the reflective wall boundary condition (\ref{eq:continBCs}). 
Thus, on reflective (symmetry) boundaries, the contributions involving $\sigma_{4,j}$ correspond to 
\[
\sum_{j=1}^d \sigma_{4,j}n_j =  - \kappa\pd{T}{n}
\]
such that ${\sigma}_{4,i}^+ = -{\sigma}_{4,i}$ imposes a zero adiabatic wall condition. We summarize this as follows:
\begin{theorem}
\label{thm:sym}
Let $\bm{g}_{\rm visc}$ denote viscous contributions from (\ref{eq:ldg1}), (\ref{eq:ldg2}), and (\ref{eq:ldg3}). Let $v_n$ be analogous to the normal velocity, such that $v_n$  and the normal stresses $\sigma_{n,j}$ are defined as
\[
v_n = \sum_{i=1}^d v_{1+i} n_i, \qquad \sigma_{n,j} = \sum_{i=1}^d\sigma_{1+i,j}n_i, \qquad i = 1,\ldots, d.
\]
In $d=2$ dimensions, if reflective (symmetry) boundary conditions are imposed by setting the exterior states for $i = 1, \ldots, d$ 
\eq{
v_{1+i}^+ &= v_{1+i} - 2v_n n_i\\
v_4^+ &= v_4\\
\sigma_{1+i,j}^+ &= 2n_i\sigma_{n,j} - \sigma_{1+i,j}\\ 
\sigma_{4,i}^+ &= -\sigma_{4,i},
}
then the viscous contribution $\bm{g}_{\rm visc}$ mimics the entropy balance such that
\[
\sum_k \LRp{\bm{g}_{\rm visc},\bm{v}}_{D^k} = - \sum_k \LRp{\sum_{i,j=1}^d \LRp{\bm{K}_{ij}\bm{\Theta}_j, \bm{\Theta}_i}_{D^k}}.
\]
Moreover, these exterior states correspond to imposing zero normal flow, zero tangential-normal stress, and zero adiabatic wall conditions. 
\end{theorem}

\subsection{Entropy dissipative boundary penalization matrices}
\label{sec:bpenmat}

Since boundary conditions are imposed weakly, it can be useful to penalize the deviation of the solution from the boundary data. To do so, we modify the penalization matrix $\bm{\tau}_{\rm visc}$ on boundary faces. The resulting matrix is non-symmetric on boundary faces in order to account for the fact that $\bm{\tau}_{\rm visc}$ now only incorporates contributions from one element, as opposed to interior interfaces which include contributions from both an element and its neighbor.

Let $\bm{\tau}_{\rm visc}$ be defined on boundary faces as
\eqlab{
\bm{\tau}_{\rm visc} = 
\tau\bmat{
0&&&\\
&-1&&\\
&&-1&\\
&\frac{\avg{v_2}}{v_4}&\frac{\avg{v_3}}{v_4}& \frac{\jump{v_4}}{2v_4}\\
}
\label{eq:penmat}
}
where $\tau \geq 0$ is a scalar penalization parameter. 

Note that division by $v_4$ is well-defined if the temperature $T > 0$ since $v_4 = -1/T < 0$. We then have the following result:
\begin{theorem}
\label{thm:pen}
Let $\bm{\tau}_{\rm visc}$ be given by (\ref{eq:penmat}). Then, the penalty term is entropy dissipative in that
\[
-\LRa{\bm{\tau}_{\rm visc}\jump{\bm{v}},\bm{v}}_{\partial \Omega} = -\frac{\tau}{2}\LRp{\LRa{\jump{v_2},\jump{v_2}}_{\partial \Omega} + \LRa{\jump{v_3},\jump{v_3}}_{\partial \Omega}+ \LRa{\jump{v_4},\jump{v_4}}_{\partial \Omega}} \leq 0.
\]
\end{theorem}
\begin{proof}
Plugging in the values for ${v}_2^+,{v}_3^+,{v}_4^+$, the penalty term reduces to
\eqlab{
-\LRa{\bm{\tau}_{\rm visc}\jump{\bm{v}},\bm{v}}_{\partial \Omega} &= -\tau\LRp{\LRa{-\jump{v_2},v_2}_{\partial \Omega} + \LRa{-\jump{v_3},v_3}_{\partial \Omega}} \label{eq:tmp1}\\
&-\tau \LRa{\jump{v_2}\frac{\avg{v_2}}{v_4} + \jump{v_3}\frac{\avg{v_3}}{v_4},v_4}_{\partial \Omega} \label{eq:tmp2} \\
&-\tau \LRa{\frac{\jump{v_4}}{v_4}\jump{v_4},v_4}_{\partial \Omega} \nonumber
}
Omitting $\tau$ for now, the final term reduces to $\LRa{\frac{\jump{v_4}}{2v_4}\jump{v_4},v_4}_{\partial \Omega} = \frac{1}{2} \LRa{\jump{v_4},\jump{v_4}}_{\partial \Omega}$, while the third term reduces to
\[
\LRa{\jump{v_2}\frac{\avg{v_2}}{v_4} + \jump{v_3}\frac{\avg{v_3}}{v_4},v_4}_{\partial \Omega} = 
\LRa{\jump{v_2},\avg{v_2}}+ \LRa{\jump{v_3},\avg{v_3}}_{\partial \Omega}.
\]
Using this, adding together (\ref{eq:tmp1}) and (\ref{eq:tmp2}) then yields
\eq{
\LRa{\bm{\tau}_{\rm visc}\jump{\bm{v}},\bm{v}}_{\partial \Omega}
&= -\tau\LRs{\LRa{\jump{v_2},\avg{v_2}-v_2}_{\partial \Omega} + \LRa{\jump{v_3},\avg{v_3}-v_3}_{\partial \Omega}}\\
&= -\frac{\tau}{2}\LRs{\LRa{\jump{v_2},\jump{v_2}}_{\partial \Omega} + \LRa{\jump{v_3},\jump{v_3}}_{\partial \Omega}}.
}
\end{proof}

\begin{remark}
When imposing heat entropy flux wall boundary conditions in Theorem~\ref{thm:adiabatic}, $\jump{v_4} = 0$ and no penalization is applied to the component $v_4 = -1/T$. 
\end{remark}

\begin{remark}
We can relate $\bm{\tau}_{\rm visc}$ to the choice of penalty matrix in \cite{dalcin2019conservative} if $\tau \propto -1/v_4 = e > 0$ and $\mu = 1, \lambda = 0$. For our numerical experiments, we choose $\tau = -\frac{1}{\LRp{\rm Re} v_4} > 0$, which mimics the scaling with respect to Reynolds number and $T$ of the penalization introduced in \cite{dalcin2019conservative}. However, a simpler choice of $\tau$ as an $O(1)$ constant does not produce significantly different results for the numerical experiments reported in this work. 
\label{remark:tau}
\end{remark}

\section{Numerical experiments}
\label{sec:num}
In this section, we present numerical experiments which verify the theoretical results proven in this work. All numerical experiments utilize the adaptive $5$th order Dormand-Prince time integration \cite{dormand1980family} to advance the solution forward in time. 

Unless specified otherwise, all numerical experiments utilize a Lax-Friedrichs penalization \cite{chen2017entropy, chan2017discretely}, where the maximum wavespeed is estimated as the maximum of the wavespeeds evaluated at the left and right states \cite{davis1988simplified}. For viscous interior dissipation, we simply take $\bm{\tau}_{\rm visc}$ to be 
\[
\bm{\tau}_{\rm visc} = \tau \bmat{0 &\\
& \bm{I}_{(d+1)\times (d+1)}},
\]
where $\tau$ is a scalar value as discussed in Remark~\ref{remark:tau}. This results in an entropy dissipation which is proportional to the norms of $\jump{v_2}, \jump{v_3}, \jump{v_4}$ over each element interface.

\subsection{Verification of accuracy}

We begin by testing convergence of the difference between the numerical solution and the imposed boundary conditions. Recall that for DG methods, the boundary conditions are imposed weakly, such that the solution does not satisfy the boundary conditions exactly. We examine convergence of the solution to zero wall boundary conditions for a simple periodic channel setup on $[-2,2]\times [-1,1]$. Periodic boundary conditions are imposed in the $x$ direction and zero adiabatic wall boundary conditions are imposed on the top and bottom walls. Simulations are run until final time $T_{\rm final} = .5$ and compute the $L^2$ error (in other words, the $L^2$ norm of the $x$ and $y$ velocities) for the initial conditions
\[
\rho = 1, \qquad u_1 = \frac{1}{10}\sin\LRp{\frac{\pi x}{2}}\cos\LRp{\frac{\pi y}{2}}, \qquad
u_2 = \frac{1}{10}\cos\LRp{\frac{\pi x}{2}}\sin \LRp{\pi y}, \qquad p = \frac{1}{{\rm Ma}^2\gamma}.
\]
We utilize ${\rm Ma} = .1$ and ${\rm Re} = 50$. The boundary penalization described in Theorem~\ref{thm:pen} is also applied. Each mesh is constructed by subdividing a quadrilateral mesh of $2K_{\rm 1D} \times K_{\rm 1D}$ elements to produce a triangular mesh. Moreover, to ensure that viscous effects near the boundary did not impact convergence, we utilized graded meshes constructed by transforming the $y$-coordinates a uniform triangular mesh via $\tilde{y} = y + .25 \sin(\pi y)$ (see Figure~\ref{fig:convergence}). 
\begin{figure}[!h]
\centering
\includegraphics[width=.33\textwidth]{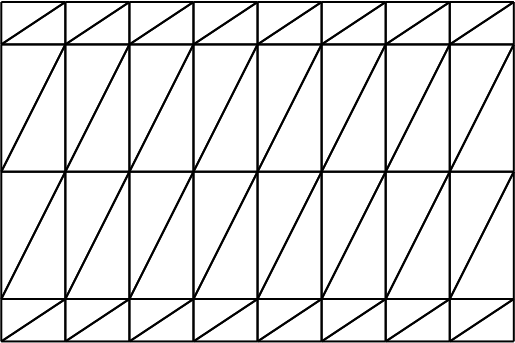}
\hspace{4em}
\includegraphics[width=.33\textwidth]{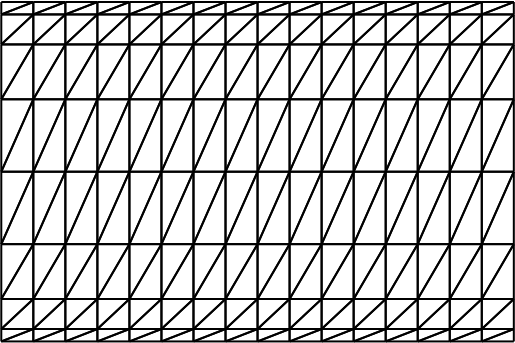}
\caption{Examples of two meshes in the sequence of meshes used for the convergence study. }
\label{fig:convergence}
\end{figure}

Table~\ref{table:convergence} shows computed errors $e_{\rm wall} = \LRp{\int_{\partial \Omega_{\rm wall}} u^2 + v^2 }^{1/2}$ for the velocity on the wall boundary $\partial \Omega_{\rm wall}$ at $y = \pm 1$. We observe asymptotic convergence rates between $O(h^{N+1})$ and $O(h^{N+2})$. We note that these rates are slightly higher than the optimal $O(h^{N+1})$ $L^2$ rate of convergence (which was observed in \cite{parsani2015entropy} for zero no-slip boundary conditions) due to the fact that the exact velocity is zero and is exactly representable by the DG approximation space. We also performed additional experiments which suggest that removing boundary penalization does not affect numerical behavior significantly, producing slightly larger errors on the coarsest meshes and roughly the same level of error on finer meshes. 
\begin{table}
\centering
\subfloat[${\rm Ma} = .1$]{
\begin{tabular}{|c||c|c||c|c||c|c||c|c||}
\hline
\diagbox[width=4em,height=2em]{$K_{\rm 1D}$}{$N$} & 1 & Rate & 2 & Rate  & 3 & Rate & 4 & Rate \\
\hline
2 & 3.97e-4 &          & 4.68e-4 &          & 4.32e-4 &         & 4.56e-4 & \\
4 & 3.45e-4 & .205  & 5.31e-4 & -.181 & 3.26e-4 & .401 & 1.08e-4 & 2.08\\
8 & 3.20e-4 & .106  & 7.30e-5 & 2.86  & 6.57e-6 & 5.63 & 6.29e-7 & 7.42\\
16 &7.74e-5 & 2.05 & 5.35e-6 & 3.77  & 1.73e-7 & 5.25 & 1.71e-8 & 5.20\\
\hline
\end{tabular}
}\\
\subfloat[${\rm Ma} = .3$]{
\begin{tabular}{|c||c|c||c|c||c|c||c|c||}
\hline
\diagbox[width=4em,height=2em]{$K_{\rm 1D}$}{$N$} & 1 & Rate & 2 & Rate  & 3 & Rate & 4 & Rate \\
\hline
2   & 8.32e-3 &          & 1.16e-2 &          & 8.81e-3 &         & 4.88e-3 & \\
4   & 6.95e-3 & .256  & 2.19e-3 & 2.41  & 2.53e-4 & 5.12 & 1.21e-4 & 5.34\\
8   & 1.13e-3 & 2.63  & 6.18e-5 & 5.14  & 1.26e-5 & 4.32 & 1.55e-6 & 6.29\\
16 & 1.97e-4 & 2.52  & 4.67e-6 & 3.73  & 4.66e-7 & 4.76 & 2.23e-8 & 6.12\\
\hline
\end{tabular}
}
\caption{$L^2$ errors for the imposition of no-slip velocity boundary conditions for ${\rm Re} = 50$ at $T_{\rm final} = 1/2$.}
\label{table:convergence}
\end{table}

\subsection{Lid-driven cavity}

We now test the imposition of viscous boundary conditions on the lid-driven cavity problem. This problem is typically used to benchmark incompressible fluid solvers \cite{ghia1982high}, though numerical experiments have also been performed for compressible flows \cite{chen1991primitive}. The domain is the bi-unit box $[-1,1]^2$, and zero no-slip conditions are imposed on the left, right, and bottom boundaries. For all experiments, we take ${\rm Ma} = .1$ and impose $u_1 = 1$ and $u_2 = 0$ on the top boundary. Initial conditions are set to be
\[
\rho = 1, \qquad u_1 = u_2 = 0, \qquad p = \frac{1}{{\rm Ma}^2\gamma}.
\]
We also augment the velocity boundary conditions with either adiabatic or isothermal temperature boundary conditions to test the new entropy stable wall boundary conditions derived in this work. 

We first consider the imposition of adiabatic boundary conditions with $g(t) = 0$. All triangular meshes are constructed by bisecting a uniform quadrilateral mesh of $K_{\rm 1D}\times K_{\rm 1D}$ elements. Figure~\ref{fig:lidcavity} shows the norm of the velocity at final time $T_{\rm final} = 100$ for ${\rm Re} = 100, 1000, 10000$. Simulations are performed using degree $N=3$ polynomials and $K_{\rm 1D} = 16$. While most solution features for ${\rm Re} = 100,1000$ are well-resolved, we note that there is under-resolution near the top left and right hand corners of the domain. This is due to the fact that the velocity boundary conditions are discontinuous between the left and right walls and the lid. However, the simulation remains stable despite this under-resolution.
\begin{figure}
\centering
\subfloat[{\rm Re} = 100]{\includegraphics[width=.3\textwidth]{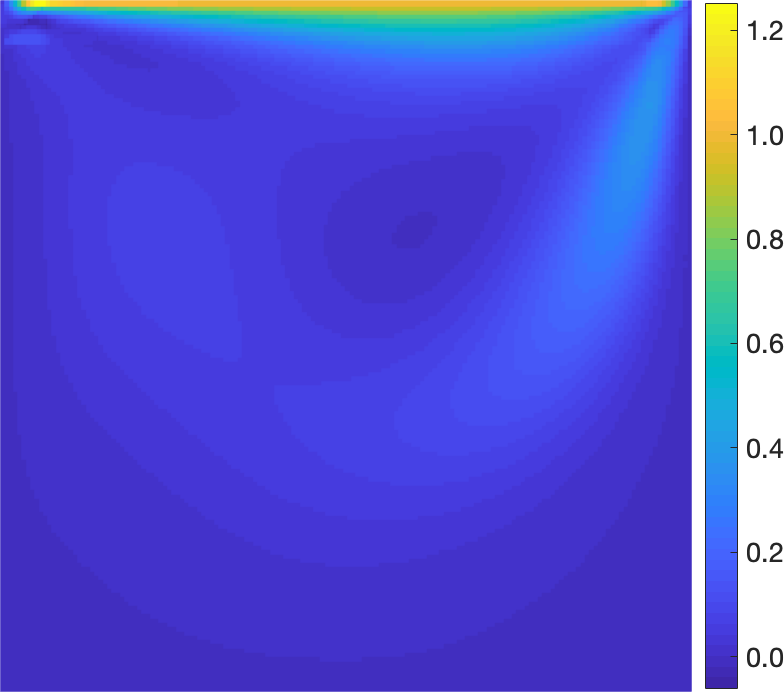}}
\hspace{1em}
\subfloat[{\rm Re} = 1000]{\includegraphics[width=.3\textwidth]{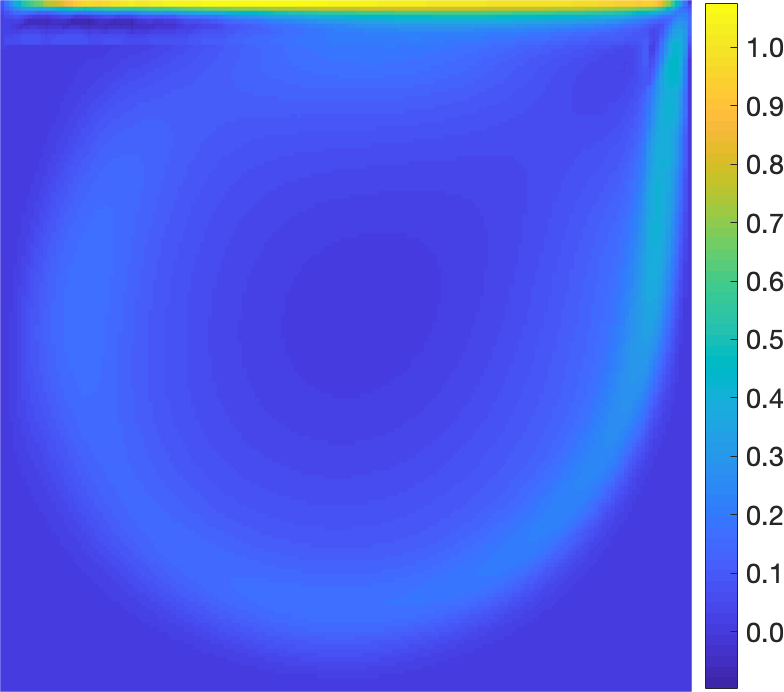}}
\hspace{1em}
\subfloat[{\rm Re} = 10000]{\includegraphics[width=.3\textwidth]{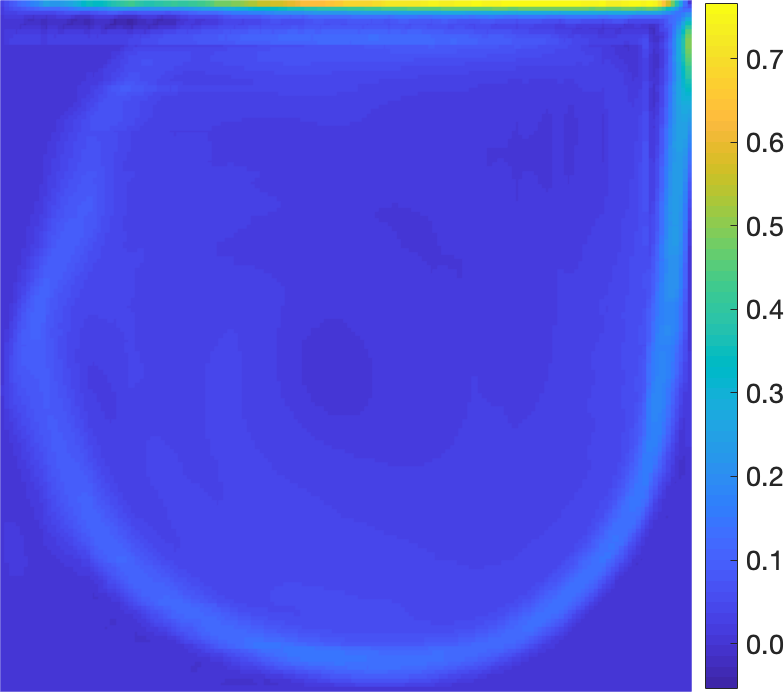}}
\caption{Norm of velocity for the lid-driven cavity problem for ${\rm Ma} = .1$ and $T_{\rm final} = 100$.}
\label{fig:lidcavity}
\end{figure}

The solutions are similar to solutions found in the literature; however, our main goal is to verify the entropy balance results proven in Theorem~\ref{thm:adiabatic}. We solve the lid-driven cavity problem with ${\rm Ma} = .1, {\rm Re} = 1000$, $N=3$, and $K_{\rm 1D} = 16$ to compute the ``viscous entropy residual'' $r(t)$ 
\eqlab{
r(t) = \sum_k \LRs{\LRp{\bm{g}_{\rm visc},\bm{v}}_{D^k} + \sum_{i,j=1}^d \LRp{\bm{K}_{ij}\bm{\Theta}_j,\bm{\Theta}_i}}.
\label{eq:viscresidual}
}
According to Theorem~\ref{thm:adiabatic}, $r(t) = \LRa{c_vg(t),1}_{\partial \Omega}$ in the absence of viscous penalization terms. Theorem~\ref{thm:pen} implies that with viscous penalization terms, $r(t)$ should be equal to $\LRa{c_vg(t),1}_{\partial \Omega}$ plus some negative semi-definite quantity which dissipates entropy. 

\begin{figure}
\centering
\subfloat[$r(t)$ over time for $g(t) = 0$]{\includegraphics[height=.22\textheight]{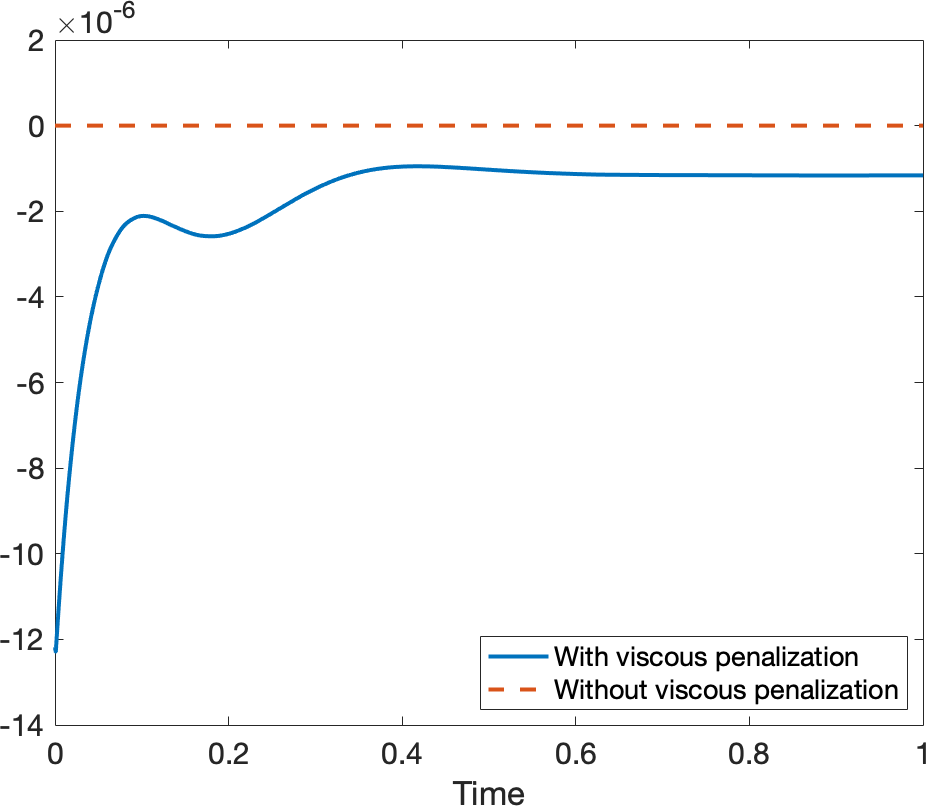}\label{subfig:rt1}}
\hspace{2em}
\subfloat[$r(t)$ for $g(t) = 10^{-4} \sin\LRp{4\pi x}$]{\includegraphics[height=.22\textheight]{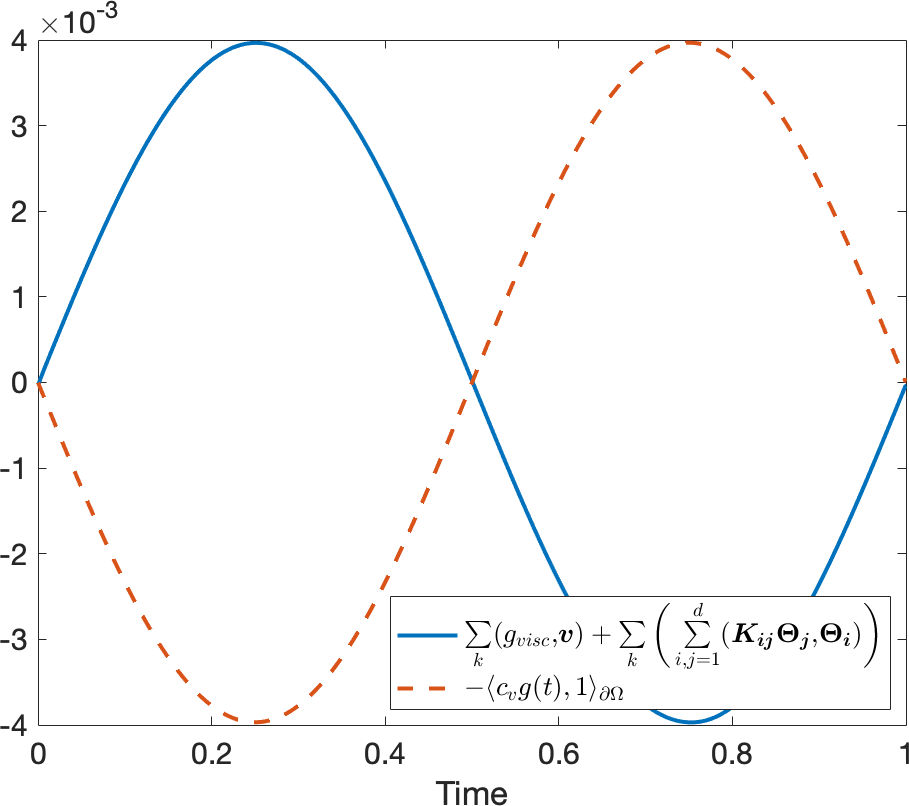}\label{subfig:rt2}}
\caption{Evolution of $r(t)$ for the lid-driven cavity under zero (adiabatic) and non-zero heat entropy flow $g(t) \neq 0$.}
\label{fig:adiabatic}
\end{figure}

Figure~\ref{subfig:rt1} shows the evolution of $r(t)$ over time for $g(t) = 0$ with and without viscous boundary penalization. Without viscous penalization, $r(t)$ is near machine precision. With viscous penalization, $r(t)$ is negative, indicating entropy dissipation. Following \cite{dalcin2019conservative}, we also consider a non-zero heat entropy flow $g(t) = 10^{-4} \sin\LRp{4\pi x}$ at the cavity lid. Here, we remove viscous penalization terms and plot both $r(t)$ and the boundary contribution $-\LRa{c_vg(t), 1}_{\partial \Omega}$. We observe that the two components are equal and opposite in sign, and adding them together yields a contribution which is again near machine precision. 

\begin{figure}
\centering
\subfloat[$r(t)$ and $\LRa{\frac{q_n}{c_v T_{\rm wall}},1}_{\partial \Omega}$]{\includegraphics[height=.22\textheight]{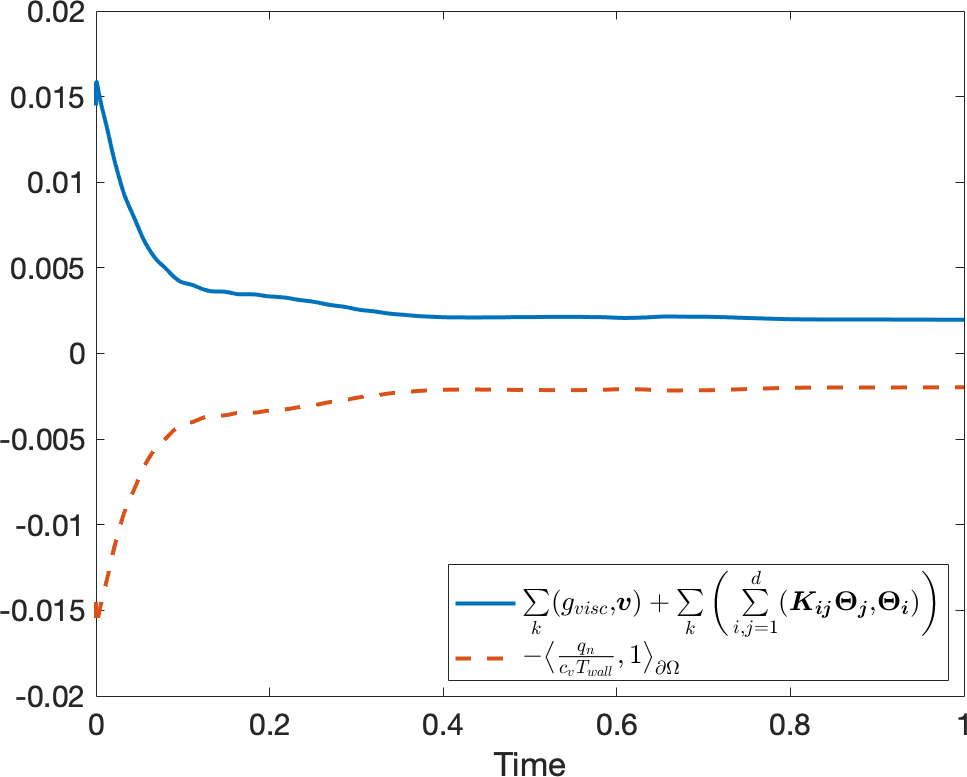}\label{subfig:rti2}}
\hspace{2em}
\subfloat[$\LRp{\bm{g}_{\rm visc},\bm{v}}_{\Omega}$]{\includegraphics[height=.22\textheight]{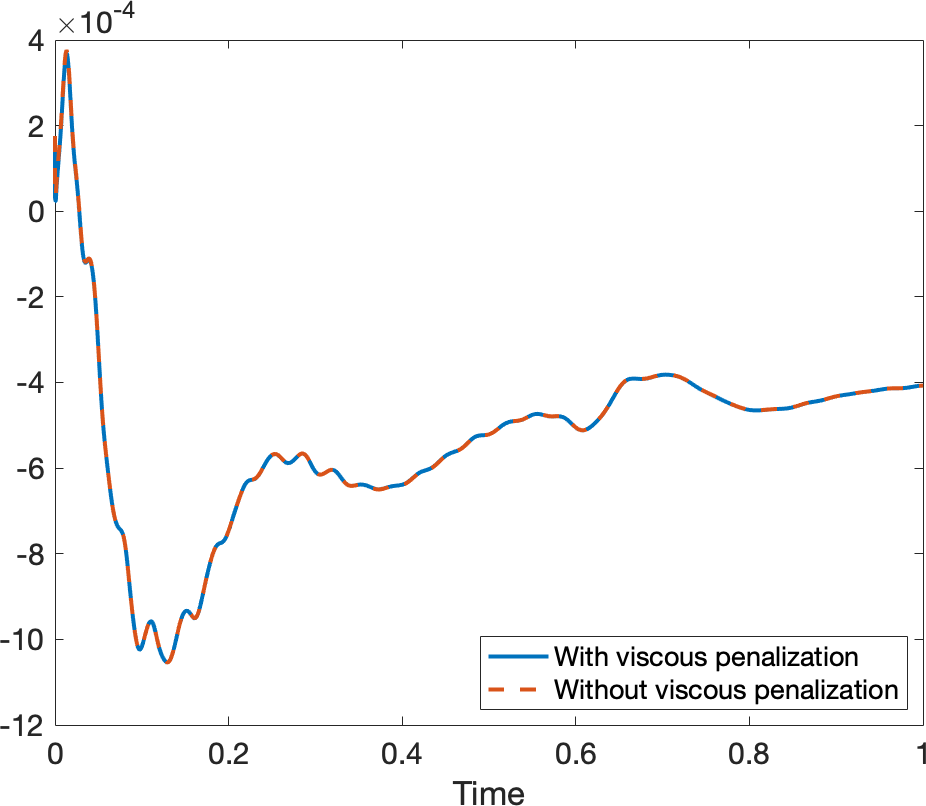}\label{subfig:rti1}}
\caption{Evolution of the viscous entropy dissipation $\LRp{\bm{g}_{\rm visc},\bm{v}}_{\Omega}$ and viscous entropy residual $r(t)$ for the lid-driven cavity under an isothermal wall boundary condition of $T = 1$.}
\label{fig:isotheraml}
\end{figure}
We now examine the imposition of fixed isothermal conditions. As noted earlier, due to the presence of an additional boundary term, this boundary condition is not provably entropy conservative. However, the formulation used in Theorem~\ref{thm:isothermal} mimics the continuous entropy inequality, and the resulting simulations appear to be remain stable in practice. We note that this mimetic property is not unique to our imposition of boundary conditions, and that the method of imposing isothermal boundary conditions in \cite{dalcin2019conservative} also semi-discretely mimics the continuous entropy inequality. We verify Theorem~\ref{thm:isothermal} using an isothermal lid-driven cavity problem with temperature $T=1$ imposed on all boundaries. The solutions at $T_{\rm final} = 100$ are nearly identical visually to the solutions in Figure~\ref{fig:lidcavity}, and are not shown for brevity. Figure~\ref{subfig:rti2} shows the evolution of the viscous entropy residual $r(t)$ and the boundary contribution $\LRa{\frac{q_n}{c_v T_{\rm wall}},1}_{\partial \Omega}$ over time. These two quantities are identical up to machine precision. Finally, Figure~\ref{subfig:rti1} shows the viscous entropy dissipation $\LRp{\bm{g}_{\rm visc},\bm{v}}_{\Omega}$. Since isothermal boundary conditions do not result in provably entropy dissipative boundary contributions, we see that this contribution is positive near the beginning of the simulation.

\subsection{Slip wall boundary conditions}

We next test the imposition of slip wall boundary conditions. We consider a channel domain $[-2,2] \times [-1,1]$ with an adiabatic no-slip wall on the bottom boundary and symmetry boundary conditions on the remaining faces of the channel. We take ${\rm Ma} = 1.5$ and ${\rm Re} = 100, 1000$ with an initial condition 
\[
\rho = \begin{cases}
5,  &x < 0\\
1,  &x \geq 0
\end{cases}, \qquad u_1 = u_2 = 0, \qquad p = \frac{1}{{\rm Ma}^2 \gamma}\rho.
\]
Figure~\ref{fig:sym} shows the squared norm of the velocity as well as the evolution of the viscous entropy residual $r(t)$ defined in (\ref{eq:viscresidual}) for a degree $N=3$ simulation. The domain is meshed using a bisected uniform quadrilateral mesh of $2 K_{\rm 1D} \times K_{\rm 1D}$ elements with $K_{\rm 1D} = 16$. No-slip wall effects are clearly visible on the bottom boundary, while the symmetry boundary condition at the top of the domain leaves the shock undisturbed in the normal direction. The viscous entropy residual is zero up to machine precision in the absence of boundary penalization, as predicted by Theorem~\ref{thm:sym}. We also observe that adding boundary penalization produces a small amount of entropy dissipation, which is more pronounced near the start of the simulation and for the under-resolved case of ${\rm Re} = 1000$.
\begin{figure}
\centering
\subfloat[Norm of velocity $u_1^2 + u_2^2$, ${\rm Re} = 100$]{\includegraphics[height=.15\textheight]{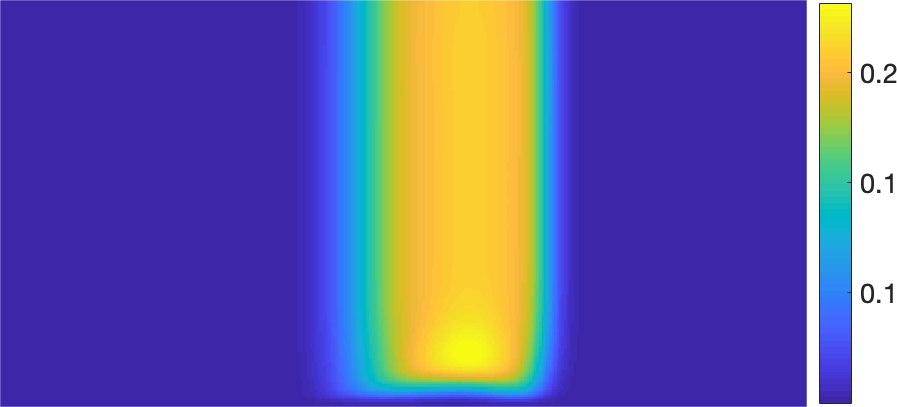}}
\hspace{2em}
\subfloat[Norm of velocity $u_1^2 + u_2^2$, ${\rm Re} = 1000$]{\includegraphics[height=.15\textheight]{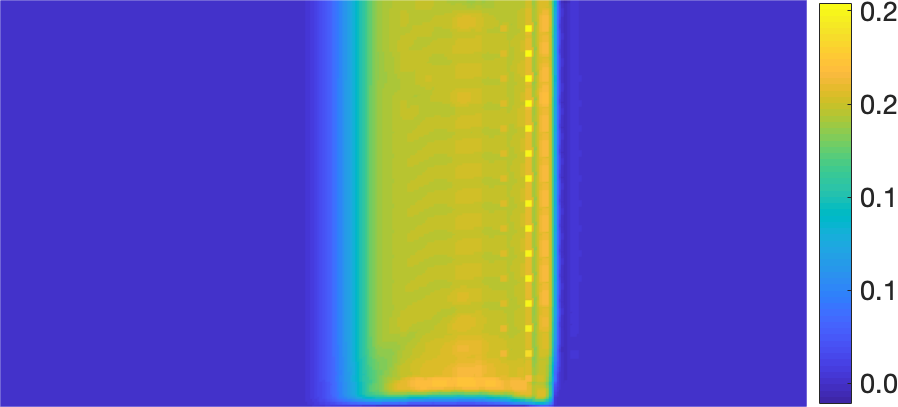}}\\
\subfloat[$r(t)$ for ${\rm Re} = 100$]{\includegraphics[height=.22\textheight]{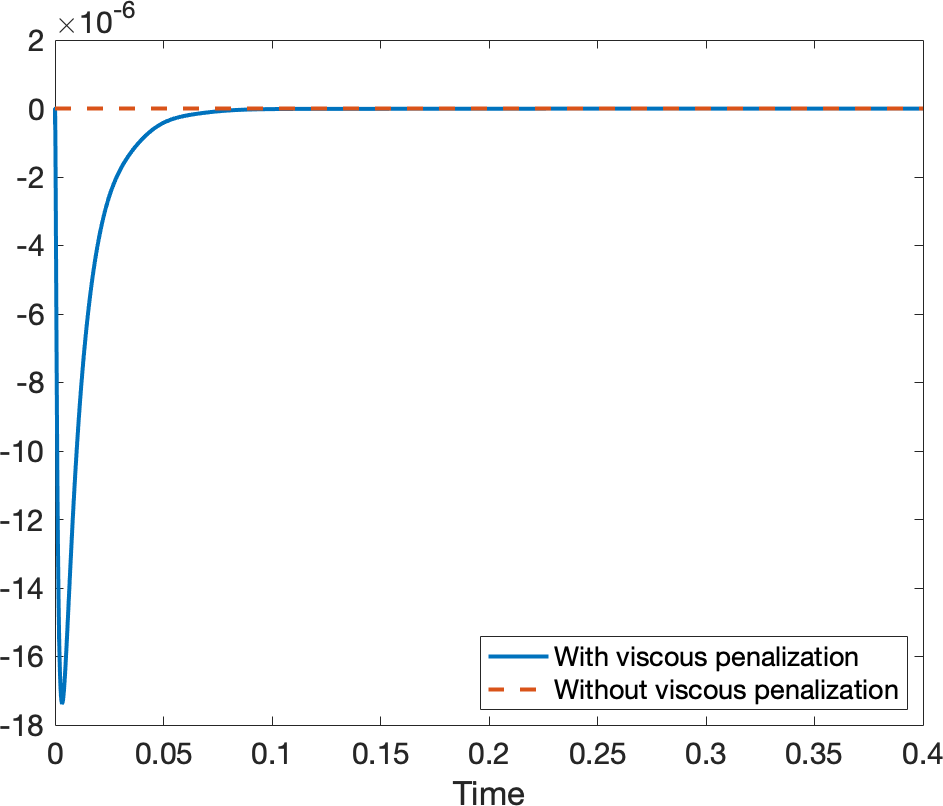}}
\hspace{6em}
\subfloat[$r(t)$ for ${\rm Re} = 1000$]{\includegraphics[height=.22\textheight]{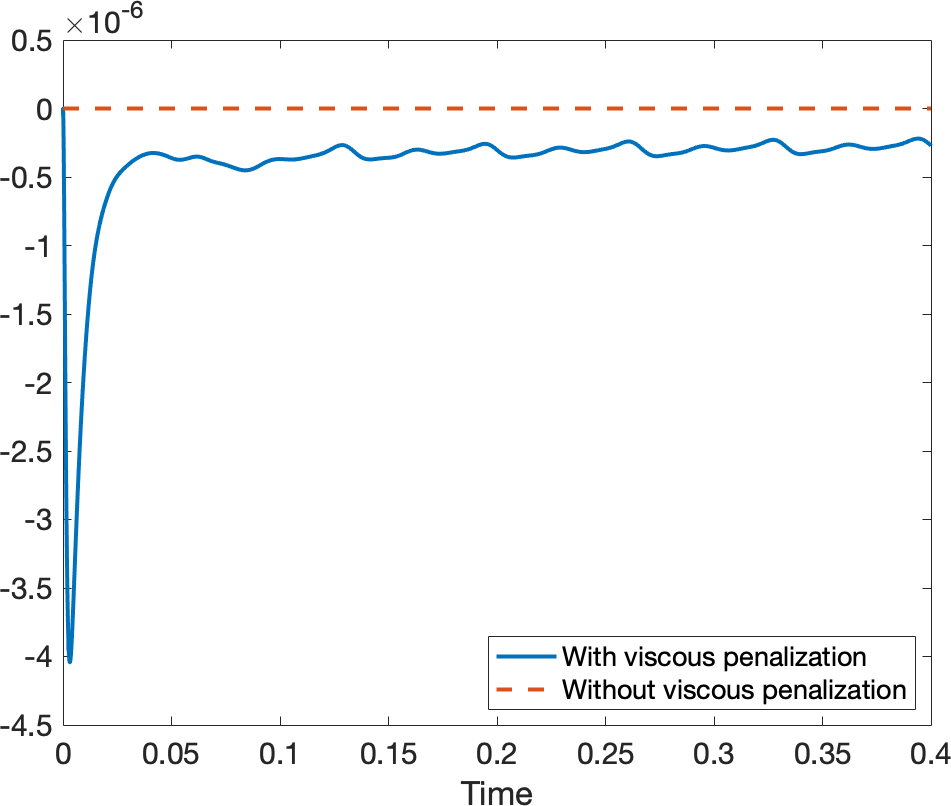}\label{subfig:rts1}}
\caption{Solutions at time $t = .4$ and evolution of the viscous entropy residual $r(t)$ over time for ${\rm Re} = 100, 1000$.}
\label{fig:sym}
\end{figure}

\subsection{Supersonic flow over a square cylinder}

We conclude by investigating supersonic flow from a square cylinder, which includes a variety of physical phenomena including shocks and vorticular features \cite{parsani2015entropy, dalcin2019conservative}. Following \cite{parsani2015entropy, dalcin2019conservative}, we take ${\rm Re} = 10^4$ and ${\rm Ma} = 1.5$ and impose zero adiabatic no-slip solid wall boundary conditions on the cylinder wall. The free-stream values are taken to be
\[
\rho = 1,\qquad  u_1 = 1, \qquad  u_2 = 0, \qquad p = \frac{1}{{\rm Ma}^2\gamma}.
\]
Both the initial condition and the exterior states on the left, top, and bottom boundaries are set using free-stream values. For the outflow boundary on the right, we utilize a simple ``extrapolation'' condition and set the exterior value equal to the interior value (we note that this is not provably entropy stable). Figure~\ref{fig:cns_cyl} shows the density for a degree $N=3$ simulation at $T_{\rm final} = 100$, as well as the triangular mesh of 16574 elements generated by Gmsh \cite{geuzaine2009gmsh}. Shocks and and trailing vortices behind the square cylinder are both visible in the numerical solution. 

\begin{figure}
\centering
\subfloat[Mesh]{\includegraphics[height=.35\textheight]{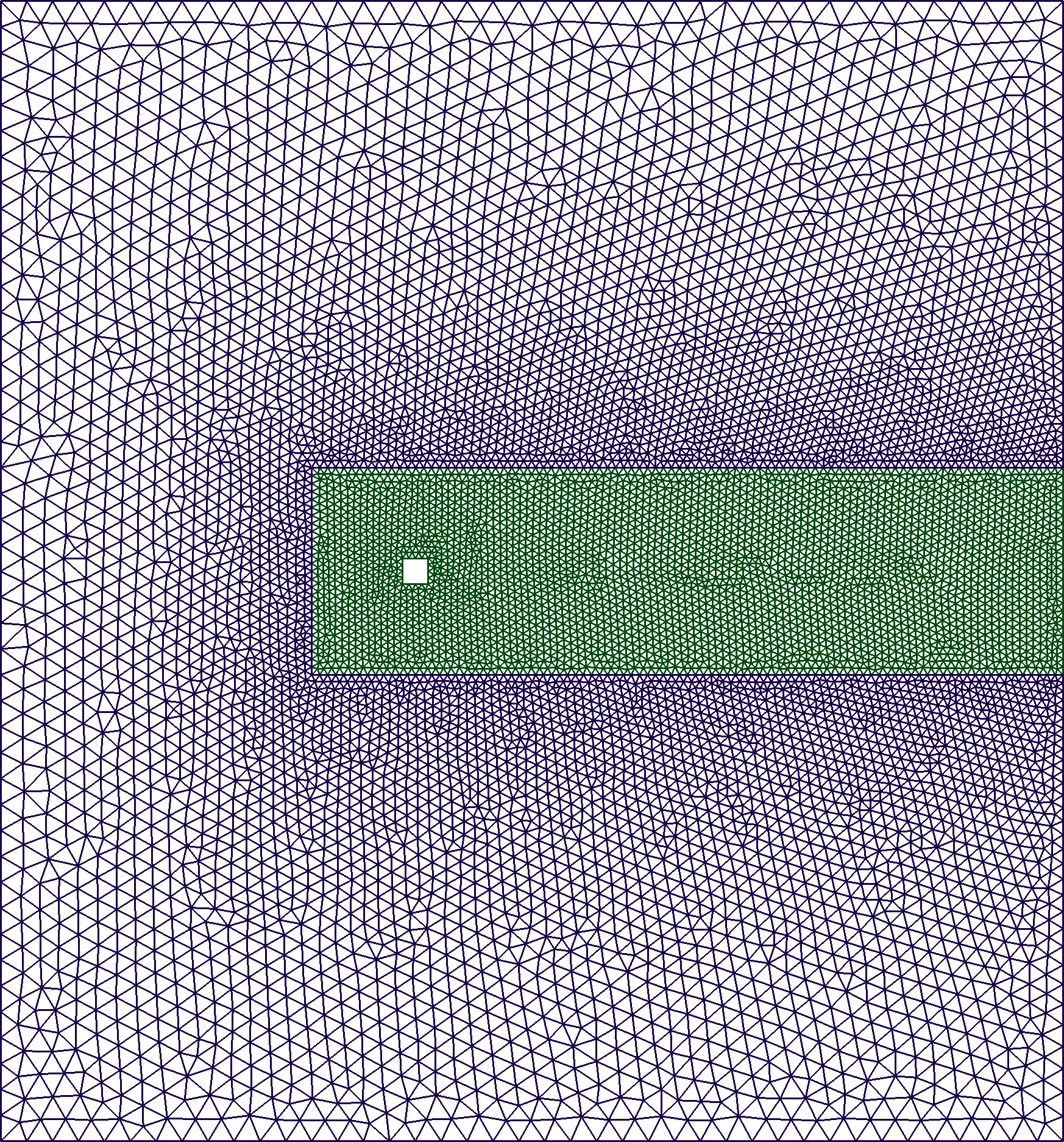}}
\hspace{3em}
\subfloat[Zoom of density $\rho$ at $T_{\rm final} = 100$]{\includegraphics[height=.35\textheight]{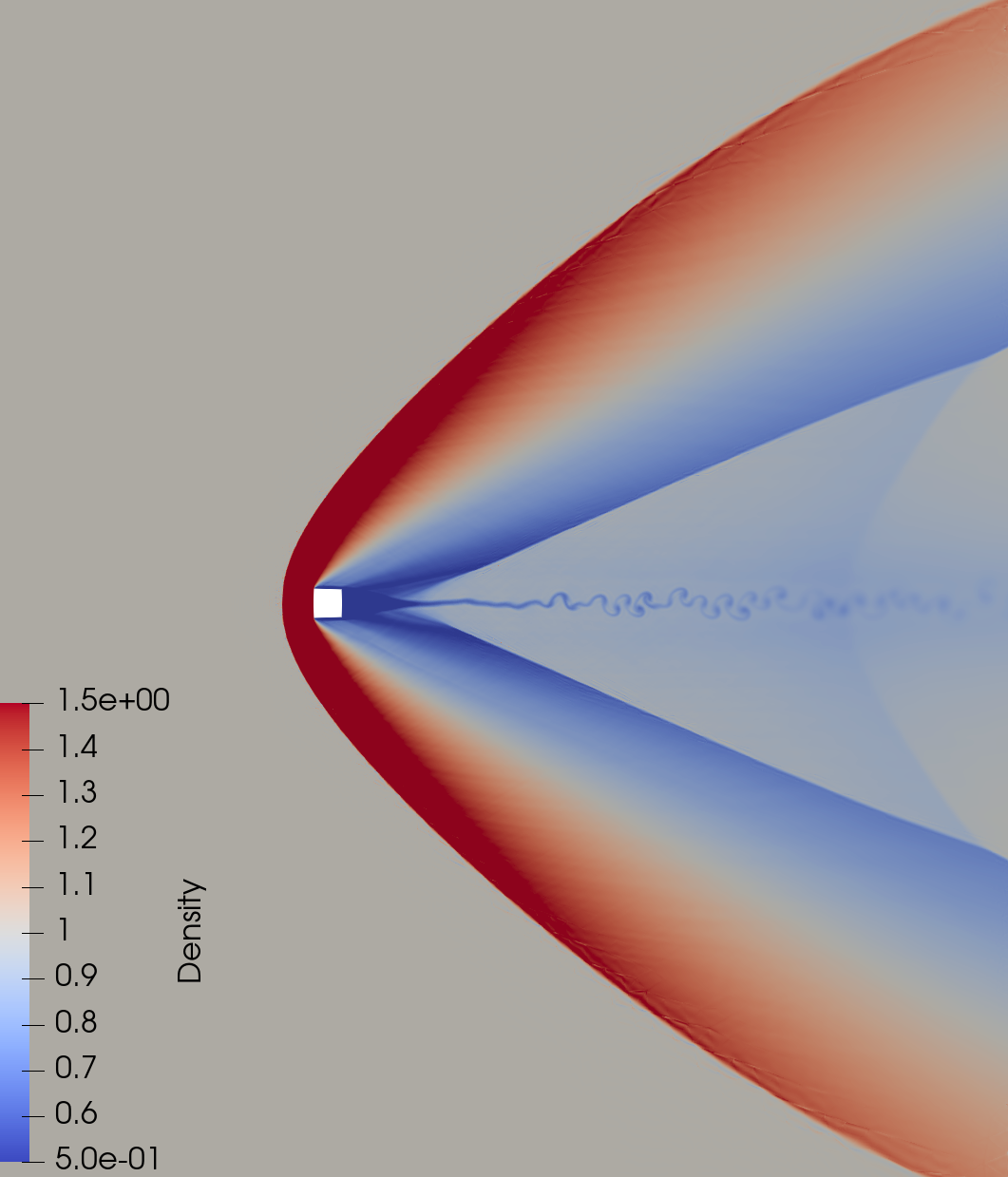}}
\caption{Computational mesh and density $\rho$ at $T_{\rm final} = 100$ using a degree $N=3$ approximation.}
\label{fig:cns_cyl}
\end{figure}

For clearer visualization, we use a color range of $[.5, 1.5]$. The simulation remains stable without additional artificial viscosity or limiting, though some numerical artifacts are observable (e.g., Gibbs oscillations in the vicinity of shock discontinuities, striations originating from the bow shock). 

\section{Conclusion}
\label{sec:conc}

In this paper, we present an entropy stable approach for discretizing viscous terms and enforcing wall boundary conditions for the compressible Navier-Stokes equations. This approach decouples the treatment of volume integrals involving symmetrized viscous coefficient matrices from the treatment of boundary terms, and results in simple and explicit formulas for the entropy stable imposition of no-slip and reflective (slip) boundary conditions. 

\section*{Acknowledgments}

Jesse Chan and Yimin Lin gratefully acknowledge support from the National Science Foundation under award DMS-CAREER-1943186. Tim Warburton was supported in part by the Exascale Computing Project, a collaborative effort of two U.S. Department of Energy organizations (Office of Science and the National Nuclear Security Administration) responsible for the planning and preparation of a capable exascale ecosystem, including software, applications, hardware, advanced system engineering, and early testbed platforms, in support of the nation’s exascale computing imperative. Tim Warburton was also supported in part by the John K. Costain Faculty Chair in Science at Virginia Tech. Finally, the authors thank Matteo Parsani and Lisandro Dalcin for informative discussions. 

\bibliographystyle{unsrt}
\bibliography{refs}

\end{document}